\documentclass[11pt]{amsart}

\usepackage[hidelinks, final]{hyperref}

\usepackage{geometry}                
\usepackage{
	amsfonts,
	amstext,
	amsmath,
	amssymb,
	amsthm,
	array,
	bm,
	color,
	enumitem,
	epigraph,
	epstopdf,
	graphicx,
	latexsym,
	mathrsfs,
	mathtools,
	mathabx,
	multirow,
	multicol,
	overpic,
	setspace,
	stmaryrd,
	subfigure,
	tikz,
	varioref,
	wrapfig}
\usepackage[all,cmtip]{xy} 
\usepackage[margin=1cm, font=footnotesize]{caption}
\usetikzlibrary{arrows}

\let\oldmarginpar\marginpar
\renewcommand\marginpar[1]{\oldmarginpar[\raggedleft\footnotesize #1]%
{\raggedright\footnotesize #1}}

\AtBeginDocument{
   \def\MR#1{}
}

\numberwithin{equation}{section}

\theoremstyle{plain}
\newtheorem{thm}{Theorem}
\numberwithin{thm}{section}
\newtheorem{lem}[thm]{Lemma}
\newtheorem{prop}[thm]{Proposition}
\newtheorem{cor}[thm]{Corollary}
\newtheorem{claim}[thm]{Claim}

\theoremstyle{definition}
\newtheorem{ex}[thm]{Example}
\newtheorem{defn}[thm]{Definition}
\newtheorem{rem}[thm]{Remark}

\newcommand{\A}{\mathcal{A}}

\newcommand{\B}{\mathcal{B}}

\newcommand{\Bi}{\PSL_2(\mathbb{Z}_{\mathbb{Q}(\sqrt{-d})})}

\newcommand{\caL}{\mathcal{L}}
\newcommand{\CC}{\mathbb{C}}
\newcommand{\ch}{\mathrm{char}}

\newcommand{\ct}{^{\dagger}}

\newcommand{\D}{\Delta}

\newcommand{\Dir}{\mathcal{D}}

\newcommand{\fu}{\mathrm{PSL}_2(\mathbb{R})}

\newcommand{\g}{\gamma}
\newcommand{\G}{\Gamma}
\newcommand{\Gal}{\mathfrak{G}}
\newcommand{\GL}{\mathrm{GL}}

\newcommand{\HH}{\mathbb{H}}
\newcommand{\HM}{\mathcal{I}}

\newcommand{\HS}{\mathcal{H}}
\newcommand{\hyp}{\frak{H}}

\newcommand{\ideal}{\vartriangleleft}

\newcommand{\Iso}{\mathrm{Isom}^+}

\newcommand{\kl}{\mathrm{PSL}_2(\mathbb{C})}

\newcommand{\M}{\mathcal{M}}

\newcommand{\mat}{\mathrm{M}}
\newcommand{\Mob}{M\"{o}bius }

\newcommand{\mtxr}{
	\begin{pmatrix}
		r & s\\
		u & v
	\end{pmatrix}}

\newcommand{\n}{\mathrm{n}}

\newcommand{\NN}{\mathbb{N}}

\newcommand{\Orb}{\mathrm{Orb}}
\newcommand{\Ord}{\mathcal{O}}

\renewcommand{\Pr}{\mathfrak{P}}
\newcommand{\Proj}{\mathrm{P}}
\newcommand{\PSL}{\mathrm{PSL}}

\newcommand{\QQ}{\mathbb{Q}}

\newcommand{\quatC}{\Big(\frac{1,1}{\mathbb{C}}\Big)}

\newcommand{\quatF}{\Big(\frac{a,b}{F}\Big)}
\newcommand{\quatFd}{\Big(\frac{a,b}{F(\sqrt{-d})}\Big)}

\newcommand{\quatK}{\Big(\frac{a,b}{K}\Big)}

\newcommand{\quatQd}{\Big(\frac{1,1}{\mathbb{Q}(\sqrt{-d})}\Big)}
\newcommand{\quatR}{\Big(\frac{1,1}{\mathbb{R}}\Big)}

\newcommand{\Ram}{\mathrm{Ram}}

\newcommand{\RR}{\mathbb{R}}

\newcommand{\SL}{\mathrm{SL}}
\newcommand{\sm}{\smallsetminus}

\newcommand{\Sp}{\mathrm{Span}}

\newcommand{\Sym}{\mathrm{Sym}}

\newcommand{\tr}{\mathrm{tr}}

\newcommand{\ZZ}{\mathbb{Z}}

\title[]
	{Macfarlane hyperbolic \boldmath$3$-manifolds}
\author{Joseph A. Quinn}
\address{Joseph A. Quinn;
	Instituto de Matem\'{a}ticas UNAM;
	Av. Universidad s/n.;
	Col. Lomas Chamilpa;
	62210 Cuernavaca, Morelos;
	M\'{e}xico}
\email{josephanthonyquinn@gmail.com}
\date{\today}

\begin{document}

\maketitle

\begin{abstract}
	We identify and study a class of hyperbolic $3$-manifolds
		(which we call Macfarlane manifolds)
		whose quaternion algebras
		admit a geometric interpretation analogous to Hamilton's
		classical model for Euclidean rotations.
	We characterize these manifolds arithmetically,
		and show that infinitely many
		commensurability classes of them arise in diverse 
		topological and arithmetic settings.
	We then use this perspective to 
		introduce a new method for computing their Dirichlet domains.
	We give similar results for a class of hyperbolic surfaces
		and explore their occurrence as subsurfaces of
		Macfarlane manifolds.
\end{abstract}

\section{Introduction}

Quaternion algebras over complex number fields
	arise as arithmetic invariants of complete
	orientable finite-volume hyperbolic $3$-manifolds
	\cite{MaclachlanReid1987}.
Quaternion algebras over totally real number fields are similarly associated
	to immersed totally-geodesic hyperbolic subsurfaces
	of these manifolds
	\cite{MaclachlanReid1987,Takeuchi1975}.
The arithmetic properties of the quaternion algebras
	can be analyzed to yield geometric and topological information
	about the manifolds
	and their commensurability classes
	\cite{MaclachlanReid2003,NeumannReid1992}.
	
In this paper
	 we introduce an alternative geometric interpretation of these algebras,
	recalling that they are a generalization of
	the classical quaternions $\HH$
	of Hamilton.
In \cite{Quinn2017a},
	the author elaborated on a classical idea of Macfarlane
	\cite{Macfarlane1900}
	to show how
	an involution on the complex quaternion algebra
	can be used to realize the action of $\Iso(\hyp^3)$
	multiplicatively,
	similarly to the classical use of the standard involution on $\HH$
	to realize the action of $\Iso(S^2)$.
Here we generalize this to a class of quaternion algebras
	over complex number fields
	and characterize them by an arithmetic condition.
We define Macfarlane manifolds as those having these algebras
	as their invariants.

We establish the existence of arithmetic and non-arithmetic
	Macfarlane manifolds
	and in each of these classes,
	infinitely many non-commensurable
	compact and non-compact examples.
We then develop a new algorithm for computing Dirichlet domains
	of Macfarlane manifolds
	and their immersed totally-geodesic hyperbolic subsurfaces,
	using the duality between
	points and isometries that comes from the quaternionic structure.

\subsection*{Main Results}

Let $X$
	be a complete orientable finite-volume hyperbolic $3$-manifold.
Let $K$
	and $\B$
	be the trace field and quaternion algebra of $X$,
	respectively.
We say that $X$
	is \emph{Macfarlane}
	if $K$
	is an imaginary quadratic extension of a (not necessarily totally)
	real field $F$,
	and the nontrivial element of the Galois group $\Gal(K:F)$
	preserves the ramification set of $\B$
	and acts on it with no fixed points.
We then show
	(in Theorem \ref{thm:macB}
	and Corollary \ref{cor:macX})
	that this is equivalent to the existence of an involution $\dagger$
	on $\B$
	which,
	with the quaternion norm,
	naturally gives rise to a $3$-dimensional
	hyperboloid $\HM_\G\subset\B$
	over $F$.
Moreover,
	$\dagger$ is unique
	and the action of $\pi_1(X)$
	by orientation-preserving isometries of $\hyp^3$
	can be written quaternionically as
	$$\pi_1(X)\lefttorightarrow\HM_\G,
		\quad (\g,p)\mapsto\g p\g^\dagger.$$

By comparison,
	via Hamilton's classical result one can use the standard involution $*$
	on $\HH$ 
	to realize $\Iso(S^2)$
	quaternionically \cite{Quinn2017a}
	as
	$$\Proj\HH^1\lefttorightarrow \HH_0^1,
		\quad (\g,p)\mapsto\g p\g^*.$$

In \S\ref{sec:man},
	we describe an adaptation of the main result for hyperbolic surfaces
	and show how,
	in certain instances,
	an immersion of a surface in a $3$-manifold
	is sufficient for the $3$-manifold to be Macfarlane.
We then study other topological and arithmetic
	conditions under which Macfarlane manifolds arise,
	culminating in the following theorem.
	
\begin{thm}\label{thm:exs}
	\begin{enumerate}
		\item[]
		\item
			Every arithmetic,
				non-compact manifold $X$
				is finitely covered by a Macfarlane manifold
				where the index of the cover is at most
				$\big|H(X,\ZZ/2)\big|$.
		\item
			Every arithmetic,
				compact manifold containing an immersed,
				closed totally geodesic subsurface is finitely covered by
				a Macfarlane manifold
				where the index of the cover is at most
				$\big|H(X,\ZZ/2)\big|$.
		\item
			Among the non-arithmetic manifolds,
				there exist infinitely many non-commensurable
				Macfarlane manifolds in each category of
				non-compact and compact.
	\end{enumerate}
\end{thm}

In \S\ref{sec:app},
	we use the quaternion model to improve on existing algorithms
	for computing Dirichlet domains
	and illustrate this with some basic examples.

\section{Preliminaries}\label{sec:bg}

See \cite{Bonahon2009, Ratcliffe1994}
	as general references
	for preliminary information on hyperbolic geometry
	and its relevance to Kleinian and Fuchsian groups.
See \cite{Quinn2017a}
	for preliminary information on algebras with involution,
	the standard Macfarlane space
	and some additional historical context.
See \cite{Voight2017}
	for a comprehensive treatment of quaternion algebras.

%

\subsection{Quaternion algebras}\label{sec:bg sub:qa}
Let $K$
	be a field with $\ch(K)\neq2$.
In this article $K$
	will usually be one of:
	$\RR, \CC$,
	a $p$-adic field,
	or a number field.
Throughout this article,
	a number field will always mean some $K\subset\CC$
	with $[K:\QQ]<\infty$
	and with a fixed embedding into $\CC$,
	i.e. a \emph{concrete number field}
	(in the sense of \cite{Neumann2011}).
	
\begin{defn}\label{defn:quats}
	Let $a,b\in K^{\times}$,
		called the \emph{structure constants}
		of the algebra.
	The \emph{quaternion algebra}
		$\quatK$
		is the associative $K$-algebra (with unity)
		$K\oplus Ki\oplus Kj\oplus Kij$,
		with multiplication rules $i^2=a$,
		$j^2=b$
		and $ij=-ji$.
	\begin{enumerate}
		\item
		The \emph{quaternion conjugate}
			of $q$
			is $q^*:=w-xi-yj-zk$.
		\item
		The \emph{(reduced) norm}
			of $q$
			is $\n(q):=qq^*=w^2-ax^2-by^2+abz^2$.
		\item
		The \emph{(reduced) trace}
			of $q$
			is $\tr(q):=q+q^*=2w$.
		\item
		$q$ is a \emph{pure quaternion}
			when $\tr(q)=0$.
	\end{enumerate}
\end{defn}

We indicate conditions on the trace via subscript,
	for instance given a subset $E\subset\quatK$,
	we write $E_0=\big\{q\in E\mid \tr(q)=0\big\}$
	and $E_+=\big\{q\in E\mid\tr(q)>0\big\}$.
We indicate conditions on the norm via superscript,
	for instance $E^1=\big\{q\in E\mid \n(q)=1\big\}$.
	
\begin{prop}\label{prop:tracenorm}\cite[\S3.3]{Voight2017}
	If $K\subset\CC$,
		then $\exists$
		a faithful matrix representation of $\quatK$
		into $\mat_2(\CC)$.
	Moreover,
		under any such representation,
		$\n$
		and $\tr$
		correspond to the matrix determinant and trace,
		respectively.
\end{prop}

We will be interested in the $K$-algebra
	isomorphism class of $\quatK$
	(preserving the embedding of $K$),
	which is not uniquely determined by the structure constants $a$
	and $b$.

\begin{thm}\cite[\S2]{MaclachlanReid2003}\samepage
	\label{thm:quatprops}
	\begin{enumerate}\label{thm:quatprops}
		\item Either $\quatK\cong\mat_2(K)$,
				or $\quatK$
				is a division algebra.
		\item $\Big(\frac{a,b}{K}\Big)\cong\mat_2(K)
				\iff\exists\ (x,y)\in K^2$
				such that $ax^2+by^2=1$.
		\item For all $x,y\in K^\times$,
			$\quatK\cong
			\Big(\frac{b,a}{K}\Big)\cong\Big(\frac{ax^2,by^2}{K}\Big)$.
	\end{enumerate}
\end{thm}

It follows that for any $K$,
	there is the quaternion $K$-algebra
	$\Big(\frac{1,1}{K}\Big)\cong\mat_2(K)$.
Moreover,
	if $\quatK$
	is not a division algebra,
	then its isomorphism class is unique and can be
	represented by $\Big(\frac{1,1}{K}\Big)$.
Thus we now focus on quaternion division algebras.

\begin{ex}
	See \cite[\S2.6 and \S2.7]{MaclachlanReid2003}
		for proofs of (3)
		and (4)
		below.
	\begin{enumerate}
		\item	Over $\RR$,
				the only quaternion division algebra
				up to isomorphism is
				$\HH:=\Big(\frac{-1,-1}{\RR}\Big)$,
				Hamilton's quaternions.
		\item There are no quaternion division algebras over $\CC$.
		\item Over each $p$-adic field,
				there is a unique quaternion division algebra
				up to isomorphism.
		\item Over each (concrete) number field,
				there are infinitely many non-isomorphic
				quaternion division algebras.\label{ex:quats item:nf}
	\end{enumerate}
\end{ex}

This raises the question of how to tell,
	when $K$
	is a number field,
	whether or not two quaternion $K$-algebras
	are isomorphic.
This can be done by investigating
	the local algebras with respect to the places of $K$,
	in the following sense.

Let $K$
	be a (concrete)
	number field and let $\B=\quatK$.
For a place $v$
	of $K$,
	let $K_v$
	be the completion of $K$
	with respect to $v$.
To each $v$,
	we associate an embedding $\sigma:K\hookrightarrow K_v$
	as described in the following paragraph,
	and then define the localization of $\B$
	with respect to $v$
	as $\B_v:=\B\otimes_\sigma K_v$.
	
If $v$
	is infinite (i.e. Archimedean),
	then it corresponds (up to complex conjugation)
	to an embedding of $K$
	into $\CC$
	fixing $\QQ$,
	under which the completion of the image of $K$
	is either $\RR$
	or $\CC$,
	and we define $\sigma$
	as the corresponding embedding.
So if $\sigma(K)\subset\RR$
	then $\B_v=\Big(\frac{\sigma(a),\sigma(b)}{\RR}\Big)$,
	which is isomorphic to either $\HH$
	or $\mat_2(\RR)$.
If $\sigma(K)\not\subseteq\RR$
	then $\B_v=\Big(\frac{\sigma(a),\sigma(b)}{\CC}\Big)$,
	which is always isomorphic to $\mat_2(\CC)$.
If $v$
	is finite (i.e. non-Archimedean),
	then it corresponds to a prime ideal $\Pr\ideal\ZZ_K$,
	where $\ZZ_K$
	is the ring of integers of $K$.
In this case we define $\sigma$
	as the identity embedding into the corresponding $p$-adic
	field $K_\Pr$,
	thus $\B_v:=\Big(\frac{a,b}{K_\Pr}\Big)$.
	
\begin{defn}\samepage
	\begin{enumerate}
		\item[]
		\item $\B$
			is \emph{ramified}
			if it is a division algebra,
			and is \emph{split}
			if $\B\cong\mat_2(K)$.
		\item $\B$
			is \emph{ramified at $v$} (respectively, \emph{split at $v$})
			if $\B_v$ is \emph{ramified} (respectively, \emph{split}).
		\item $\Ram(\B)$
			is the set of real embeddings and prime ideals
			that correspond to the places where $\B$
			is ramified.
	\end{enumerate}
\end{defn}

The set $\Ram(\B)$
	provides the desired classification of isomorphism classes
	of quaternion algebras over number fields,
	as follows.

\begin{thm}\samepage\cite[\S14.6]{Voight2017}\label{thm:ram}
	\begin{enumerate}
		\item $\B$ is split if and only if $\Ram(\B)=\O$.
		\item $\Ram(\B)$
			uniquely determines the isomorphism class of $\B$.
		\item $\Ram(\B)$
			is a finite set of even cardinality,
			and every such set of places of $K$
			occurs as $\Ram(\B)$
			for some $\B$.
	\end{enumerate} 
\end{thm}

\subsubsection{\bf Involutions on Quaternion Algebras}

We will have need to think about involutions on $\B$
	besides quaternion conjugation,
	so we include some basic notions of more general such involutions.
	
\begin{defn}\label{defn:involution}\samepage
	An \emph{involution}
		on $\B$
		is a map $\star :\B\rightarrow\B: x\mapsto x^\star $
		such that $\forall\  x,y\in\B:$
	\begin{enumerate}
		\item $(x+y)^\star =x^\star +y^\star $,
		\item $(xy)^\star =y^\star x^\star $, and
		\item $(x^\star )^\star =x$.
	\end{enumerate}
\end{defn}

\begin{defn}\samepage
	For a subset $E\subset\B$,
		the \emph{set of symmetric elements of $\star $ in $E$},
		is $\mathrm{Sym}(E,\star):=\{x\in E\mid x^\star =x\}$.
\end{defn}
\begin{defn}
	\begin{enumerate}
		\item[]
		\item $\star $
			is \emph{of the first kind}
			if $K=\mathrm{Sym}(K,\star)$.
		\item $\star $
			is \emph{standard}
			if it is of the first kind and $\forall\  x\in A:xx^\star \in K$.
		\item $\star $
			is \emph{of the second kind}
			if $K\neq\mathrm{Sym}(K,\star)$.
	\end{enumerate}
\end{defn}

\begin{prop}\cite[\S I.2]{KnusEtc1998}\label{prop:involutions}
	Let $\star$
		be an involution on $\B$.
	\begin{enumerate}
		\item $1^\star =1$.
		\item If $\star $
			is of the first kind then it is $K$-linear.
		\item If $\star $
			is of the second kind then
			$\big[K:\mathrm{Sym}(K,\star)\big]=2$.
	\end{enumerate}
\end{prop}

\subsection{The arithmetic of hyperbolic 3-manifolds}
\label{sec:bg sub:arith}

Let $X$
	be a complete orientable finite-volume hyperbolic $3$-manifold,
	and from here on this is what we will mean
	when we say \emph{manifold}.
Then $\pi_1(X)\cong\G<\kl$
	for some discrete group $\G$
	(i.e. $\G$
	is a Kleinian group).
Let $\widehat\G:=\big\{\pm\g\mid\{\pm\g\}\in\G\big\}<\SL_2(\CC)$.

\begin{defn}\label{defn:tq}
	\begin{enumerate}
		\item[]
		\item The \emph{trace field
				of $\G$}
				is $K\G:=
					\QQ\big(\big\{\tr(\g)\mid\g\in\widehat\G\big\}\big)$.
		\item The \emph{quaternion algebra 
				of $\G$}
				is $B\G:=\big\{\sum_{\ell=1}^n t_\ell\g_\ell\mid
					t_\ell\in K\G, \g_\ell\in\widehat\G, n\in\NN\big\}$.
	\end{enumerate}
\end{defn}

\begin{rem}
	In the literature
		these are usually denoted by $k_0\G$
		and $A_0\G$,
		but we write them differently
		to avoid confusion with the notation for pure quaternions.
\end{rem}

Then $K\G$
	is a concrete number field which,
	recall,
	means it is a fixed subfield of $\CC$,
	and $B\G$
	is a quaternion algebra over $K\G$ \cite{MaclachlanReid1987}.
By Mostow-Prasad rigidity,
	these are manifold invariants in the sense that
	if $\G$
	and $\G'$
	are two discrete faithful representations of $\pi_1(X)$,
	then $K\G=K\G'$
	and $B\G\cong B\G'$
	via a $K\G$-algebra isomorphism
	(though the converse does not hold).
So we may also refer to them as the \emph{trace field}
	and \emph{quaternion algebra} of $X$
	up to homeomorphism.
	
\begin{defn}\label{defn:itq}\samepage
	Let $\G^{(2)}:=\langle\g^2\mid\g\in\G\rangle$.
	\begin{enumerate}
		\item The \emph{invariant trace field
				of $\G$}
				is $k\G:=K\G^{(2)}$.
		\item The \emph{invariant quaternion algebra 
				of $\G$}
				is $A\G:=B\G^{(2)}$.
	\end{enumerate}
\end{defn}

These likewise are invariants of $X$,
	but have the stronger property of being commensurability invariants.
That is,
	if $\G$
	is commensurable
	up to conjugation
	to some Kleinian group $\G'$,	
	then $k\G=k\G'$
	and $A\G\cong A\G'$
	(though the converse does not hold)
	\cite{NeumannReid1992}.

We call $X$
	\emph{arithmetic}
	if $\G$
	is an arithmetic group in the sense of
	\cite{BorelHarish-Chandra1962},
	but this admits the following alternative characterization
	which is more useful for our purposes.

\begin{defn}\label{defn:arith}\cite{MaclachlanReid1987}
	\begin{enumerate}
		\item $\G$ (or $X$)
			is
			\emph{derived from a quaternion algebra}
			if there exists a quaternion algebra $\B$
			over a field $K$
			with exactly one complex place $\sigma$,
			such that $\B$
			is ramified at every real place of $K$,
			and $\exists$
			an order $\Ord\subset\B$
			such that $\G$
			is isomorphic to a finite-index subgroup of $\Proj\Ord^1
				:=\Ord^1/\pm1$.
		\item $\G$ 
			(or $X$)
			is \emph{arithmetic}
			if it is commensurable
			up to conjugation
			to one that is derived from a
			quaternion algebra.
	\end{enumerate}
\end{defn}

If $\G$
	is derived from a quaternion algebra $\B$
	over a field $K$,
	then $K=K\G=k\G$
	and $\B\cong B\G\cong A\G$.
If $\G$
	is arithmetic,
	then $\G^{(2)}$
	is derived from a quaternion algebra.
In general,
	$\G^{(2)}$
	is a finite-index subgroup of $\G$.
\cite{NeumannReid1992}

While $k\G$
	and $A\G$
	are generally more suitable to the application of arithmetic,
	we will work instead with $B\G$
	so that we may take advantage of the natural embedding
	$\widehat\G\hookrightarrow B\G$.
(To simplify notation,
	and where it will not cause confusion,
	we will often refer to an element $\{\pm\g\}\in\widehat\G$
	by a representative $\g\in\G$.)
Often, $A\G$
	and $B\G$
	coincide (though not always \cite{Reid1990}).
	
\begin{prop}\samepage
	\begin{enumerate}\label{prop:A=B}
		\item[]
		\item $k\G=K\G$
			if and only if $A\G\cong B\G$.
		\item If $\HS^3/\G$
			is a knot or link complement,
			then $A\G\cong B\G$.
	\end{enumerate}
\end{prop}

\begin{proof}
	We prove (1),
		and see \cite[\S4.2]{MaclachlanReid2003}
		for (2).
	The reverse implication is immediate.
	For the forward implication,
		note that $A\G\subset B\G$
		and both are $4$-dimensional vector spaces,
		so if they are over the same field
		then they must be the same.
\end{proof}

We now collect some important properties of these invariants.
	
\begin{thm}\cite[\S8.2]{MaclachlanReid2003}\label{thm:invprops}
	\begin{enumerate}
		\item If $X$
				is non-compact,
				then $B\G\cong\Big(\frac{1,1}{K\G}\Big)$
				and $A\G\cong\Big(\frac{1,1}{k\G}\Big)$.
		\item If $X$
				is non-compact and arithmetic,
				then $\exists d\in\NN$
				such that $k\G=\QQ(\sqrt{-d})$.
		\item If $X$
				is compact and arithmetic,
				then $A\G$
				is a division algebra.
	\end{enumerate}
\end{thm}	

\subsection{Dirichlet domains}\label{sec:bg sub:DD}
\cite[\S7.4]{Bonahon2009}
We conclude our preliminary discussion by introducing
	the type of fundamental domain we will be interested in
	for our application in \S\ref{sec:app}.
Let $\mathcal{X}$
	be a model for hyperbolic $3$-space
	upon which $\G$
	acts faithfully.
\begin{defn}\label{defn:dir}
	The \emph{Dirichlet domain}
		for the action of $\G$
		on $\mathcal{X}$
		centered at $c$
		is
	\begin{gather*}
		\Dir_\G(c):=
			\Big\{p\in\mathcal{X} \; \Big|\
			\forall\  \g\in \G\smallsetminus\mathrm{Stab}_\G(c):
			d(c,p)\leq d\big(c,\g(p)\big)\Big\}.
	\end{gather*}
\end{defn}

As long as $\mathrm{Stab}_\G(c)=\{1\}$,
	the Dirichlet domain $\Dir_\G(c)$
	is a fundamental domain for the action of $\G$
	on $\mathcal{X}$,
	and since for us $\G$
	is torsion-free,
	this the case for all $c$.
There is a more explicit characterization of $\Dir_\G(c)$,
	for which the following notation is useful.

\begin{defn}\label{defn:sides}
	For each $\g\in\G$,
		let
	\begin{enumerate}
		\item
			$g(\g)$
				be the geodesic segment from $c$
				to $\g(c)$,
		\item
			$\widetilde{s}(\g)$
				be the complete geodesic hyperplane
				perpendicularly bisecting $g(\g)$,
				and
		\item 
			$E(\g)\subset\mathcal{X}$
				be the half-space
				$\Big\{p\in\mathcal{X} \; \Big|\:
					d(c,p)\leq d\big(c,\g(p)\big)\Big\}$
				(so $\widetilde{s}(\g)=\partial E(\g)$).
	\end{enumerate}
\end{defn}

Since $\G$
	is geometrically finite,
	there is some finite minimal set $S\subset\G$
	such that
	$$\Dir_\G(c)=\underset{\g\in S}{\bigcup}E(\g).$$
We say that $\g$
	\emph{contributes a side}
	to $\Dir_\G(c)$
	if $\g\in S$
	and for each of these,
	let $s(\g):=\widetilde{s}(\g)\cap\partial\Dir_\G(c)$,
	which we call the \emph{side contributed by $\g$}.
The idea is to understand $X$
	by studying $\Dir_\G(c)$
	equipped with side-pairing maps on its boundary.
These side pairings are given by applying $\g^{-1}$
	to the side contributed by $\g$,
	for each $\g\in S$.
	


%
%

\section{Macfarlane Quaternion Algebras and $\Iso(\hyp^3)$}
	\label{sec:main}

Our goal in this section is to show that the arithmetic definition of
	Macfarlane manifolds
	admits the geometric interpretation of containing a hyperboloid model
	tailor-made for the action of $\G$,
	made precise in Theorem \ref{thm:macB}.

\begin{defn}\label{defn:macB}\samepage
	A quaternion algebra $\B$
		over a fixed field $K\subset\CC$
		is \emph{Macfarlane}
		if
	\begin{enumerate}
		\item
			$\exists F\subset\RR$
				and $\exists d\in F^+$
				such that $K=F(\sqrt{-d})$, and
		\item
			the nontrivial element $\sigma$
				of $\Gal(K:F)$
				preserves $\Ram(\B)$
				and $\forall v\in\Ram(\B)$,
				$\sigma(v)\neq v$.
	\end{enumerate}
	The manifold $X$
		with corresponding Kleinian group $\G$
		is \emph{Macfarlane}
		if $B\G$
		is Macfarlane.
\end{defn}

\begin{rem}
	Note that $F$
		is a fixed subfield of $\RR$,
		thus $F^+:=\{f\in F\mid f>0\}$
		is well-defined.
	In the case where $F$
		is a (recall, concrete) number field,
		we are only concerned with the positivity of its elements
		under the identity embedding into $\RR$,
		i.e. we form $F^+$
		the same way regardless of whether or not $F$
		is totally real.
\end{rem}

\begin{ex}\samepage
	\begin{enumerate}
		\item[]
		\item $\quatC$
				is Macfarlane because $\CC=\RR(\sqrt{-1})$
				and $\Ram\quatC=\O$.
		\item The figure-$8$ knot complement
				and its quaternion algebra
				$\Big(\frac{1,1}{\QQ(\sqrt{-3})}\Big)$
				are Macfarlane.
		\item The quaternion algebra $\B$
				over $\QQ(\sqrt{-5})$	
				with $\Ram(\B)=\big\{(3,1+\sqrt{-5}),(3,1-\sqrt{-5})\big\}$
				is Macfarlane
				because $\sigma$
				permutes the ramified places.
		\item The quaternion algebra
				$\B=\Big(\frac{\sqrt2,\sqrt2}{\QQ\sqrt{1-\sqrt2}}\Big)$
				is Macfarlane.
			To see this,
				take $F=\QQ(\sqrt2)$,
				$d=\sqrt2-1$,
				and notice that $\Ram(\B)$
				consists of the pair of conjugate real embeddings
				that take $\sqrt{1-\sqrt2}$
				to $\pm\sqrt{1+\sqrt2}$.
	\end{enumerate}
\end{ex}

We now state our main result.

\begin{thm}\label{thm:macB}
	$\B$
		is Macfarlane if and only if it
		admits an involution $\dagger$
		such that $\Sym(\B,\dagger)$
		(which we denote by $\M$),
		equipped with the restriction of the quaternion norm,
		is a quadratic space of signature $(1,3)$
		over $\Sym(K,\dagger)$.
		
	Moreover,
		$\dagger$
		is unique and,
		letting $\M_+^1=\big\{p\in\M\mid\tr(p)>0,\n(p)=1\big\}$,
		a faithful action of $\Proj\B^1$
		upon $\hyp^3$
		by orientation-preserving isometries
		is defined by the group action
		$$\mu_\B:\Proj\B^1\lefttorightarrow\M_+^1,
			\quad(\g,p)\mapsto\g p\g^\dagger.$$
\end{thm}

\begin{rem}
The isomorphism class of $\B$,
	as a quaternion algebra over the concrete number field $K$,
	does not include non-identity embeddings $K\hookrightarrow\CC$.
Thus the signature of $\Sym(\B,\dagger)$
	over $\Sym(K,\dagger)$
	is well-defined as long as $\Sym(K,\dagger)$
	is real.
\end{rem}
\begin{defn}\samepage
	$\M$
		as in Theorem \ref{thm:macB}
		is called a \emph{Macfarlane space}.
\end{defn}

\subsection*{Proof of Theorem \ref{thm:macB}}

First we show that the existence of an involution as in the Theorem
	is equivalent to a condition on the field and structure constants 
	(as in Definition \ref{defn:quats})
	of the algebra,
	up to isomorphism.
Recall from Theorem \ref{thm:quatprops}
	that even over a fixed field $F(\sqrt{-d})$,
	the structure constants $a$
	and $b$
	are not unique up to $F(\sqrt{-d})$-algebra isomorphism,
	nor are their signs.
The idea is to show that there is a representative of this isomorphism
	class which normalizes the Macfarlane space to a convenient form.

\begin{lem}\label{lem:Hilbert}
	$\B$
		admits an involution with the properties described in
		Theorem \ref{thm:macB}
		if and only if $\B\cong\Big(\frac{a,b}{F(\sqrt{-d})}\Big)$
		for some $F\subset\RR$
		and $a,b,d\in F^+$.
\end{lem}

The reverse direction of this,
	in the case where $\B=\quatFd$,
	is Theorem 7.2 of \cite{Quinn2017a}.
This generalizes to $\B\cong\quatFd$
	because an isomorphism
	between quaternion algebras is also a quadratic space isometry
	with respect to the quaternion norms \cite[\S5.2]{Voight2017},
	thus it transfers the multiplicative structure,
	the involution and the Macfarlane space.

So it suffices to prove the forward direction,
	and we do this via a series of claims.
Let $\B$
	be a quaternion algebra over a field $K$
	and suppose $\B$
	admits an involution $\dagger$
	with the properties described in Theorem \ref{thm:macB}.
	
\begin{claim}
	$K$
		is of the form $F(\sqrt{-d})$
  	  	where $F=\Sym(K,\dagger)\subset\RR$
		and $d\in F^+$,
		and $\dagger|_K$
		acts as complex conjugation.
\end{claim}

\begin{proof}
	If $K$
		were real,
		then $\n$
		would be a quadratic form of signature $(2,2)$,
		making it impossible for $\B$
		to contain a subspace of signature $(1,3)$,
		thus $K\not\subseteq\RR$.
	On the other hand,
		for a space to have nontrivial signature over $\Sym(K,\dagger)$,
		we must have $\Sym(K,\dagger)\subset\RR$.
	This means $\dagger$
		is an involution of the second kind,
		which implies $[K:\Sym(K,\dagger)]=2$
		by Proposition \ref{prop:involutions},
		as desired.
		
	We now show that $\dagger|_K$
		acts as complex conjugation.
	Since $-d\in F=\Sym(K,\dagger)$,
		we have
		$$(\sqrt{-d}\ct)^2=(\sqrt{-d}^2)\ct=(-d)^\dagger=-d,$$
		thus $\sqrt{-d}\ct=\pm\sqrt{-d}$.
	Since $\sqrt{-d}\notin\Sym(K,\dagger)$,
		this leaves $\sqrt{-d}\ct=-\sqrt{-d}$.
\end{proof}

Write $\M=\Sym(\B,\dagger)$.	
We are going to use the signature of $\n|_\M$
	to prove that $\B$
	has real structure parameters up to isomorphism,
	but a priori we do not know what $\M$
	is.
So we will first need to establish that $\M$
	includes enough linearly independent elements of $\B$,
	in the following sense.

\begin{claim}\label{lem:span}
	$\Sp_K(\M)=\B$.
\end{claim}

\begin{proof}
	We know that $F\subset\M$
		and $\Sp_K(F)=K$,
		so it suffices to prove $\Sp_K(\M_0)=\B_0$.
		
	Let $E=\{s_1,s_2,s_3\}$
		be a basis for $\M_0$
		over $F$
		and assume by way of contradiction that $E$
		is not linearly independent over $K$.
	Then $\exists\  k_\ell\in K$
		such that $\sum_{\ell=1}^3k_\ell s_\ell=0$.
	Since $K=F(\sqrt{-d})$,
		we have that each $k_\ell=f_{\ell,1}+f_{\ell,2}\sqrt{-d}$
		for some $f_{\ell,1},f_{\ell,2}\in F$.
	Substituting these into $\sum_{\ell=1}^3k_\ell s_\ell=0$
		and rearranging terms,
		we get
	\begin{align*}	
		f_{1,1}s_1+f_{2,1}s_2+f_{3,1}s_3&
			=-\sqrt{-d}(f_{1,2}s_1+f_{2,2}s_2+f_{3,2}s_3).
	\end{align*}
	But $f_{1,1}s_1+f_{2,1}s_2+f_{3,1}s_3$
		and $f_{1,2}s_1+f_{2,2}s_2+f_{3,2}s_3$
		both lie in $\M$,
		so are fixed by $\dagger$,
		meanwhile by the previous claim,
		$\sqrt{-d}\ct=-\sqrt{-d}$.
	So applying $\dagger$
		to both sides of the equation gives
	\begin{align*}	
		f_{1,1}s_1+f_{2,1}s_2+f_{3,1}s_3
				&=\sqrt{-d}(f_{1,2}s_1+f_{2,2}s_2+f_{3,2}s_3).
	\end{align*}	
	Adding the last two displayed
		equations then gives that
		$f_{1,2}s_1+f_{2,2}s_2+f_{3,2}s_3=0$.
	Since $f_{1,2},f_{2,2},f_{3,2}\in F$,
		this contradicts that $E$
		is a basis for $\M_0$
		over $F$.
		
	We conclude that $E$
		is linearly independent over $K$,
		giving
	\begin{gather*}
		\dim_K\big(\Sp_K(E)\big)=\dim_K\big(\Sp_K(\M_0)\big)=3,
	\end{gather*}
		which forces $\Sp_K(\M_0)=\B_0$.
\end{proof}

\begin{claim}
	$\B\cong\quatFd$
		for some $a,b\in F^+$.
\end{claim}

\begin{proof}
	The norm $n|_\M$
		is a real-valued quadratic form of signature $(1,3)$,
		so there exists an orthogonal basis $D$
		for $\M$
		so that the Gram matrix for $\n|_\M$
		with respect to $D$
		is a diagonal matrix $G_{\n|_\M}^D$,
		with diagonal of the form $(f_1,-f_2,-f_3,-f_4)$
		for some $f_\ell\in F^+$.
	Since $\Sp_K(\M)=\B$,
		this same $D$
		is also an orthogonal basis for $\B$
		over $K$.
		
	Let $C$ be the standard basis $\{1,i,j,ij\}$
		for $\B$.
	Then $C$ is another orthogonal basis for $\B$
		over $K$
		and,
		in particular,
		the Gram matrix $G_\n^C$
		for $\n$
		with respect to $C$
		is the diagonal matrix with diagonal $(1,-a,-b,ab)$.
		
	Now while $G_{\n|_\M}^D$
		and $G_{\n}^C$
		are not congruent over $F$,
		they are congruent over $K$
		because $D$
		and $C$
		are both bases for $\B$,
		i.e. $\exists\ \delta\in\GL_4(K)$
		such that
	\begin{align*}\label{gramstand}
		\delta G_\n^D\delta^\top=G_\n^C.
	\end{align*}
	But since $G_\n^D$
		and $G_\n^C$
		are diagonal and nonzero on their diagonals,
		$\delta$ must also be diagonal and nonzero on its diagonal,
		i.e. $\exists\ x_\ell\in K^\times$
		such that $\delta$
		is the diagonal matrix with diagonal $( x_1,x_2,x_3,x_4)$.
	Plugging in to the last displayed equation
		and solving for the $f_\ell$
		gives
	\begin{gather*}
		f_1=\frac{1}{x_1^2}\qquad
		f_2=\frac{a}{x_2^2}\qquad
		f_3=\frac{b}{x_3^2}\qquad
		f_4=\frac{-ab}{x_4^2}.
	\end{gather*}
	
	Now let $\B'=\Big(\frac{f_2,f_3}{F(\sqrt{-d})}\Big)$
		and recall that $f_2,f_3\in F^+$.
	Then $\B$
		has the desired form because,
		by Theorem \ref{thm:quatprops},
		$$\B'\cong\Big(\frac{f_2x_2^2,f_3x_3^2}{F(\sqrt{-d})}\Big)=\B.$$
\end{proof}

This completes the proof of Lemma \ref{lem:Hilbert}.

Now to complete the proof of Theorem \ref{thm:macB},
	we show that the condition on the isomorphism class
	of the symbol $\quatK$
	from Lemma \ref{lem:Hilbert}
	is equivalent to the arithmetic characterization of Macfarlane
	quaternion algebras given by Definition \ref{defn:macB}.

\begin{lem}\label{lem:macram}
	Let $\B$
		be a quaternion algebra over $K=F(\sqrt{-d})$
		where $F\subset\RR$
		and $d\in F^+$.
	The nontrivial element of $\Gal(K:F)$
		preserves $\Ram(\B)$
		with no fixed points if and only if $\exists a,b\in F^+$
		such that $\B\cong\quatK$.
\end{lem}

\begin{proof}
	With $a, b, F$
		and $K$
		as in the statement,
		notice that $\quatK=\quatF\otimes_FK$.
	Also if $a$ (or $b$)
		is negative,
		then by Theorem \ref{thm:quatprops},
		we can replace it by $-ad$
		(or $-bd$)
		without changing the isomorphism class.
	So it suffices to prove
		that the condition on $\Ram(\B)$
		is equivalent to the existence of a quaternion algebra $\A$
		over $F$
		such that $\B\cong\A\otimes_FK$.

	If there is such an $\A$,
		then $\Ram(\B)$
		is the set of pairs of real
		(respectively, non-Archimedean)
		places $v,v'$
		of $K$
		associated with a real place in $\Ram(\A)$
		that splits into two real places of $K$
		(respectively, associated with a prime ideal of $\ZZ_F$
		that splits in $\ZZ_K$).
	By Theorem \ref{thm:quatprops} (3),
		this sets up a bijection between quaternion $K$-algebras
		of the form $\A\otimes_FK$
		and sets of pairs of places arrising as described.
	In this situation,
		the nontrivial element $\sigma$
		of $\Gal(K:F)$
		will stabilize $\Ram(\B)$,
		and interchange $v$
		and $v'$.
	Conversely,
		if $\sigma$
		preserves and acts freely on $\Ram(\B)$,
		then there are no inert or ramified primes as these
		are fixed by $\Gal(K:F)$,
		forcing the ramified places of $K$
		to be as described above.
\end{proof}

This completes the proof of Theorem \ref{thm:macB}.
The following consequence has computational advantages
	which will be exploited in \S\ref{sec:app}.

\begin{cor}\label{cor:abd}
	If $\B$
		is Macfarlane,
		then there is an isomorphism $\B\cong\quatFd$
		where $a,b,d\in F^+$.
	In this case the Macfarlane space is
		$$\M=F\oplus Fi\oplus Fj\oplus\sqrt{-d}Fij$$
		and for $q=w+xi+yj+zij\in\B$
		with $w,x,y,z\in F(\sqrt{-d})$,
		the involution $\dagger$
		is given by
	\begin{align}\label{dagger}
		q\ct=\overline{w}+\overline{x}i+\overline{y}j-\overline{z}ij.
	\end{align}
\end{cor}

A final remark on Theorem \ref{thm:macB}
	is that even though we are using $\M_+^1$
	as a hyperboloid model for the group action,
	it is technically not a model for hyperbolic $3$-space
	unless $F=\RR$.
In the cases of interest,
	$F$
	is a number field embedded in $\RR$,
	and this gives us all we need to model a Macfarlane manifold.
If a complete model for hyperbolic $3$-space were desired,
	one is given by $(\M\otimes_F\RR)_+^1$,
	but this would lose the arithmetic structure that makes
	Macfarlane manifolds interesting.
		
\section{Macfarlane Manifolds}\label{sec:man}

In this section we explore the various conditions
	under which Macfarlane manifolds arise.
First we clarify how Theorem \ref{thm:macB}
	translates to the context of Kleinian groups.
Then we look at an adaptation of our results to hyperbolic surfaces
	and see some settings in which hyperbolic subsurfaces imply
	that manifolds are Macfarlane.
The remainder of the section culminates in a proof of Theorem \ref{thm:exs},
	about the diverse settings in which Macfarlane manifolds arise.

As in \S\ref{sec:bg sub:arith},
	let $X$
	denote a complete orientable finite-volume hyperbolic $3$-manifold
	with Kleinian group $\G\cong\pi_1(X)$.
Let $K=K\G$
	and $\B=B\G$.

\begin{defn}\samepage
	When $X$
		is Macfarlane and $\M\subset\B$
		is its Macfarlane space as in Theorem \ref{thm:macB},
		define $\HM_\G:=\M_+^1$
		and call this a \emph{quaternion hyperboloid model for $\G$
			(or $X$)}.
\end{defn}

It is immediate that $\HM_\G$,
	up to quadratic space isometry over $K\G$,
	is a manifold invariant.
		
By the definition of $B\G=\B$,
	there is no confusion in speaking of $\G$
	quaternionically,
	as lying in $\Proj\B^1$
	rather than in $\kl$.
In this way,
	$\G$
	(up to choice of representatives in $\widehat\G$)
	and $\HM_\G$
	are both subsets of $\B$,
	making sense of the following.

\begin{cor}\label{cor:macX}
	If $X$
		is Macfarlane,
		then the action of $\G$
		by orientation-preserving isometries of $\hyp^3$
		is faithfully represented by
		$$\mu_\G:\G\lefttorightarrow\HM_\G,
			\quad(\g,p)\mapsto\g p\g^\dagger.$$
\end{cor}

\subsection{Hyperbolic Surfaces and Subsurfaces}
The author showed in \cite{Quinn2017a}
	how the representation of $\Iso(\hyp^3)$
	in $\quatC$
	restricts to a representation of $\Iso(\hyp^2)$
	in $\quatR$.
Similarly,
	we seek analogues of Theorem \ref{thm:macB}
	and Corollary \ref{cor:macX}
	for hyperbolic surfaces.
This is possible to some extent
	but we must proceed with care because of the following important
	differences between the $3$-dimensional
	and $2$-dimensional settings.
For a complete orientable finite-volume hyperbolic surface $S$,
	the group $\pi_1(S)$
	admits discrete faithful representations
	into $\fu$
	but,
	in the absence of Mostow-Prasad rigidity,
	the trace of some fixed hyperbolic element of $\pi_1(S)$
	could take any value in $\RR^{>2}$
	under these representations,
	so the trace field is no longer a manifold invariant.
We resolve this by requiring $S$
	to have a fixed immersion into a hyperbolic $3$-manifold
	$X$
	under which it is totally geodesic,
	which implies an injection $\pi_1(S)\hookrightarrow\pi_1(X)$.
We can do this in such a way that,
	taking the discrete faithful representation
	$\pi_1(X)\cong\Gamma<\PSL_2(\CC)$,
	we represent $\pi_1(S)\cong\Delta<\G$
	as a fixed group of matrices.
	
We then define the \emph{(invariant) trace field of $\D$}
	and \emph{(invariant) quaternion algebra of $\D$}
	on this fixed representation
	in the same way as in Definitions \ref{defn:tq}
	and \ref{defn:itq},
	and we denote them similarly by ($k\D$) $K\D$
	and ($A\D$)
	$B\D$,
	respectively.
These now have properties similar
	to what we saw in the Kleinian setting:
	$K\Delta$
	and $k\Delta$
	are now fixed subfields of $\RR$,
	$B\Delta$
	and $A\D$
	are quaternion algebras over $K\D$
	and $k\D$
	respectively \cite{Takeuchi1969},
	and $A\D$
	is a commensurability invariant
	\cite[\S4.9 and \S5.3.2]{MaclachlanReid2003}.
	
The results from \cite[\S6]{Quinn2017a}
	along with the proof of Lemma \ref{lem:Hilbert}
	then give the corollary below,
	after the following observations.
The field $K\D$
	is real,
	and so now the involution $\dagger$
	is of the first kind.
That is, $\Sym(K\D,\dagger)=K\D$
	and $\Sym(B\D,\dagger)$
	is comprised of $K\D$
	and the unique $2$-dimensional
	negative-definite subspace with respect to the norm on $B\D$.

\begin{cor}\label{cor:surf}
	If $B\D\cong\Big(\frac{a,b}{K\D}\Big)$
		for some $a,b>0$,
		then it admits an involution $\dagger$
		such that $\Sym(B\D,\dagger)$
		(which we denote by $\caL$),
		equipped with the restriction of the quaternion norm,
		is a quadratic space of signature $(1,2)$
		over $K\D$.
		
	Moreover,
		$\dagger$
		is unique and,
		letting
		$\caL_+^1=\big\{p\in\caL\mid\tr(p)>0,\n(p)=1\big\}$,
		a faithful action of $\D$
		upon $\hyp^2$
		by orientation-preserving isometries
		is defined by the group action
		$$\mu_\D: \D\lefttorightarrow\caL_+^1,
			\quad(\g,p)\mapsto\g p\g^\dagger.$$
\end{cor}

\begin{defn}
	We call the space $\caL\subset B\D$
		as above
		a \emph{restricted Macfarlane space},
		and we call the space $\HM_\D:=\caL_+^1$
		a \emph{quaternion hyperboloid model for $\D$}.
\end{defn}
	
We next give a way of realizing Macfarlane $3$-manifolds
	using totally-geodesic subsurfaces.
There is a stronger version of this in the arithmetic setting,
	which will be done in the next subsection.
	
\begin{prop}\label{prop:subsurf}\samepage
	If $X$
		contains an immersed closed totally-geodesic surface,
		and its trace field is $F(\sqrt{-d})$
		for some $F\subset\RR$
		(not necessarily totally real)
		and $d\in F^+$
		(under the fixed embedding),
		then $X$
		is Macfarlane.
\end{prop}
	
\begin{proof}
	Let $S\subset X$
		be a surface as in the hypothesis.
	Then $\pi_1(S)$
		has a Fuchsian representation $\D<\fu$
		and $\pi_1(X)$
		has a Kleinian representation $\G<\kl$
		such that $\D<\G$.
	Then $K\D\subset K\cap\RR$
		is a totally real subfield.
	Therefore $B\D\subset B\G$
		is a quaternion subalgebra over a subfield of $F$.
	Hence $\exists\ a,b\in F$
		so that $B\D=\Big(\frac{a,b}{K\D}\Big)$.
	Then $B\D\otimes_{K\D}K\G
		=\Big(\frac{a,b}{F(\sqrt{-d})}\Big)\subset B\G$,
		and since $B\G$
		is a $4$-dimensional
		vector space over the same field,
		we have $B\G=\Big(\frac{a,b}{F(\sqrt{-d})}\Big)$.
	Thus
		$X$
		is Macfarlane by Lemma \ref{lem:macram},
		where recall,
		as in the proof of that Lemma,
		if $a$
		(or $b$)
		is negative,
		we can replace it with $-da$
		(or $-db$).
\end{proof}

\begin{rem}
	With an immersion as above,
		the action $\mu_\G$
		as given in Corollary \ref{cor:macX}
		restricts to the action $\mu_\D$
		as given in Corollary \ref{cor:surf}.
	An example using this will be shown
		in \S\ref{sec:app sub:ex}.
\end{rem}

\subsection{Arithmetic Macfarlane Manifolds}
	\label{sec:macs sub:non-compact}

In this subsection,
	we construct examples of arithmetic Macfarlane manifolds
	and prove parts (1) and (2)
	of Theorem \ref{thm:exs}.
	
Pending the following arguments,
	notice that if $X$
	is the initial arithmetic manifold,
	the bound on the index of its cover by a Macfarlane manifold is at most
	$\big|H(X,\ZZ/2)\big|$
	because this is the index of the group $\Gamma^{(2)}$
	in $\Gamma$
	and,
	as discussed in \S\ref{sec:bg sub:arith},
	the invariant quaternion algebra of $\G$
	is the quaternion algebra of $\Gamma^{(2)}$.
The rest of the argument is separated into
	the non-compact and compact cases.
	
\subsubsection{\bf Non-compact Arithmetic Macfarlane Manifolds}
\label{sec:macs sub:non-compact sub:arith}
Recall from Theorem \ref{thm:invprops}
	that a non-compact manifold has a split quaternion algebra,
	thus the structure parameters can always be taken as $a=b=1$.
So the manifold will be Macfarlane provided it satisfies our condition
	on the trace field.
In the arithmetic case,
	this is easy to control.

\begin{lem}\label{lem:non-comp arith}
	Every arithmetic non-compact manifold is finitely covered by
		a Macfarlane manifold.
\end{lem}
	
\begin{proof}
	Let $\G$
		be arithmetic and non-cocompact.
	Then $\G$
		is commensurable to a Bianchi group
		$\PSL_2(\ZZ_{\QQ(\sqrt{-d})})$
		($d\in\NN$ squarefree),
		which has quaternion algebra $\quatQd$
		by Theorem \ref{thm:invprops}
		(see also \cite[\S8.2]{MaclachlanReid2003}).
	This is Macfarlane by Lemma \ref{lem:Hilbert}.
\end{proof}

If $\G$
	additionally is derived from a quaternion algebra,
	then $B\G$
	will also be $\quatQd$
	and it will be Macfarlane itself.
Otherwise we can only say that $A\G=\quatQd$,
	which is why above we work up to commensurability.
It is very common for these groups to satisfy $K\G=k\G$
	(so that $B\G=A\G$,
	making them Macfarlane),
	but this is not always the case \cite{Reid1990}.
For example,
	the arithmetic manifold m009
	in the cusped census has invariant trace field $\QQ(\sqrt{-7})$
	but its trace field is $\QQ\Big(\sqrt{\frac{5-\sqrt{-7}}{2}}\Big)$,
	which contains no degree $2$
	subfield with a real place
	(making it not Macfarlane).

\begin{ex}\samepage
	Arithmetic link complements are Macfarlane
		by Proposition \ref{prop:A=B}.
	\begin{enumerate}
		\item The figure-8 knot complement is Macfarlane,
				with trace field $\QQ(\sqrt{-3})$.
		\item The Whitehead link is Macfarlane,
				with trace field $\QQ(\sqrt{-1})$.
		\item The six-component chain link is Macfarlane,
				with trace field $\QQ(\sqrt{-15})$.
	\end{enumerate}
	For more information on arithmetic link complements,
		see \cite{Hatcher1983,BakerReid2002,
			GrunewaldSchwermer1981}
		and \cite[\S9.2]{MaclachlanReid2003}.
\end{ex}

\subsubsection{\bf Compact Arithmetic Macfarlane Manifolds}

Quaternion algebras of compact manifolds can be split
	or can be among the infinitely many ramified possibilities
	over each possible trace field.
If we look at groups derived from quaternion algebras,
	we can easily construct examples of compact arithmetic
	Macfarlane manifolds in the usual way that these groups are formed.
	
Start with a concrete number field $K=F(\sqrt{-d})$
	with $F\subset\RR$,
	$d\in F^+$,
	where $K$
	has a unique complex place.
Choose a pair $a,b\in K\cap\RR$
	such that $\quatK$
	is ramified,
	take an order $\Ord\subset\quatK$
	and let $\G=\Proj\Ord^1$,
	$X=\mathcal{I}^3/\G$.
Naturally,
	not every isomorphism class of quaternion algebras over $K$
	admits such a choice of $a,b$
	but there are infinitely many that do.
For instance,
	for each split non-Archimedean place of $F$
	that splits into a pair $v_1,v_2$
	of places of $K$
	(of which there are infinitely many,
	by the Chebotarev density theorem),
	choose the quaternion algebra $B$
	over $K$
	with $\Ram(B)=\{v_1,v_2\}$.
Thus there are infinitely many non-commensurable
	compact arithmetic Macfarlane manifolds over each such $K$.
	
We now look at a geometric condition that gives compact arithmetic
	Macfarlane manifolds.

\begin{lem}\label{lem:comp arith}
	If $X$
		is arithmetic,
		compact
		and contains an immersed closed totally-geodesic surface,
		then $X$
		is finitely covered by a Macfarlane manifold.
\end{lem}

\begin{proof}
	Let $S\subset X$
		be a surface as in the hypothesis.
	Then $\pi_1(S)\cong\Delta'<\Gamma$
		for some Fuchsian group $\Delta'$.
	Then by Theorem 9.5.2 of \cite{MaclachlanReid2003},
		$\Delta'<\Delta$
		for some arithmetic Fuchsian group $\Delta$
		satisfying $[k\Gamma:k\Delta]=2$
		and $k\Delta=k\Gamma\cap\RR$.
	This implies that $K\Gamma^{(2)}=K\Delta^{(2)}(\sqrt{-d})$
		for some $d\in K\Delta^{(2)}$.
	Lastly,
		we know that $\G^{(2)}$
		is a finite index subgroup of $\G$,
		i.e. $X$
		finitely covers $\HS^3/\G^{(2)}$,
		which contains the immersed closed totally geodesic surface
		corresponding to $\Delta'^{(2)}$,
		so apply Proposition \ref{prop:subsurf}
		to $\Gamma^{(2)}$.
\end{proof}

\subsection{Non-arithmetic Macfarlane manifolds}

We now complete the proof of Theorem \ref{thm:exs}
	by providing an abundance of commensurability classes
	of non-arithmetic examples.

\begin{lem}\label{lem:non-comp non-arith}
	There are infinitely many commensurability classes
		of non-compact non-arithmetic Macfarlane manifolds.
\end{lem}
\begin{proof}
	Recall that the quaternion algebra of a non-compact manifold
		is necessarily split (Theorem \ref{thm:invprops}),
		so it suffices to satisfy the condition on the trace field.
		
	In \cite{ChesebroDebois2014},
		an infinite class of non-commensurable link complements
		are generated,
		all having invariant trace field $\QQ(\sqrt{-1},\sqrt{2})$,
		by gluing along totally-geodesic $4$-punctured spheres.
	Since this field is not of the form $\QQ(\sqrt{-d})$
		with $d\in\NN$
		and these manifolds are non-compact,
		they are non-arithmetic.
	Since they are link complements,
		by Proposition \ref{prop:A=B}
		their trace fields are also $\QQ(\sqrt{-1},\sqrt{2})$,
		which is of the desired form.
\end{proof}


To realize non-arithmetic compact Macfarlane manifolds,
	we start some with arithmetic compact
	Macfarlane manifold with immersed totally-geodesic subsurfaces.
Then we apply the technique of interbreeding
	introduced by Gromov and Piatetski-Shapiro
	\cite{GromovPiatetski-Shapiro1987}.
This entails gluing together a pair of
	non-commensurable arithmetic manifolds
	along a pair of totally-geodesic and isometric subsurfaces,
	resulting in a non-arithmetic manifold.
We use a variation on this technique introduced by Agol \cite{Agol2006}
	called inbreeding,
	whereby one glues together a pair of geodesic subsurfaces
	bounding non-commensurable submanifolds
	of the same arithmetic manifold.

\begin{lem}\label{lem:non-arith comp}
	For every arithmetic Macfarlane manifold $X$
		containing an immersed closed
		totally-geodesic surface,
		there exist infinitely many commensurability classes of
		non-arithmetic Macfarlane manifolds
		with the same quaternion algebra as $X$.
\end{lem}
\begin{proof}
	Let $X$
		be as in the statement
		and $\G\cong\pi_1(X)$
		be a Kleinian group.
	Then $X$ contains infinitely many commensurability classes of
		immersed closed totally-geodesic arithmetic
		hyperbolic subsurfaces \cite[\S9.5]{MaclachlanReid2003}.
	Since Kleinian groups are LERF
		(locally extended residually finite) \cite{Agol2013},
		among these are infinitely many pairs of surfaces
		where each pair is isometric but noncommensurable
		\cite[\S4]{Agol2006}.	
	Let $S_1$
		and $S_2$
		be one of these pairs.
	Since these are arithmetic,
		they each correspond to a lattice in a quaternion subalgebra
		over a real subfield of $F$,
		and these lattices are noncommensurable
		because of the noncommensurability of $S_1$
		with $S_2$.

	Deform the lattices so that $S_1$
		and $S_2$
		can be glued together via the identification map
		$f:S_1\rightarrow S_2$
		as in \cite{Agol2006}
		and form the amalgamated product $X_{1,2}:=X\ast_fX$.
	Since we did this by a deformation and gluing of noncommensurable
		lattices,
		$X_{1,2}$
		will have a systole inducing a matrix
		of non-integral trace in any discrete faithful
		Kleinian representation of $\G_{1,2}:=\pi_1(X_{1,2})$.
	Thus $X_{1,2}$	
		is non-arithmetic.
	However the reflection involution through the identified subsurface
		in $X_{1,2}$
		lies in the commensurator of $\G$.
	So we have $K\G=K\G_{1,2}$
		and $B\G\cong B\G_{1,2}$
		over $K\G$.
		
	Lastly,
		by \cite[\S4]{Agol2006},
		there exists an infinite sequence of choices
		for pairs of surfaces $S_\ell, S_m$
		as above so that the injectivity radius of $X_{\ell,m}$
		gets arbitrarily small.
	But by Margulis' Lemma,
		there is a lower bound to the injectivity radius
		of any class of commensurable
		non-arithmetic manifolds.
	Thus in our sequence $\{X_{\ell,m}\}$,
		we enter a new commensurability class infinitely often
		as this radius approaches $0$,
		giving infinitely many non-commensurable non-arithmetic
		manifolds all with quaternion algebra $B\G$
		(up to $K\G$-algebra isomorphism).
\end{proof}

\section{Applications}\label{sec:app}

In this section,
	we give an example of a computational
	advantage of the quaternion hyperboloid model.
To encourage further exploration of this and other potential applications,
	we first provide a way of transferring data between the current approach
	and the more conventional upper half-space model.
Then we introduce a tool that improves algorithms
	for finding Dirichlet domains 
	of Macfarlane manifolds.

Throughout this section,
	let $X$
	be Macfarlane.
Choose $\G\cong\pi_1(X)$
	so that $\B:=B\G=\quatFd$,
	$F\subset\RR$
	and $a,b,d\in F^+$
	(guaranteed by Corollary \ref{cor:abd}),
	then identify $\G$
	with its natural image in $\Proj\B^1$.
Let $\dagger$
	be the unique involution on $\B$
	so that $\Sym(\B,\dagger)$ is the Macfarlane space
	$\M$.
Then
$$\M=F\oplus Fi\oplus Fj\oplus\sqrt{-d}Fij$$
	and $\HM_\G=\M_+^1$
	is a hyperboloid model for $X$,
	and recall the group action from Corollary \ref{cor:macX}:
	$$\mu_\G:\G\times\HM_\G, (\g,p)\mapsto\g p\g\ct.$$

\subsection{Comparison with the Upper Half-Space Model}
%

The following is an adaptation of
	\cite[Theorem 5.2]{Quinn2017a},
	so we omit the computational details.
	
The map
\begin{align}\label{rhoB}
	\rho_{\B}:\B\hookrightarrow
		\mat_2\big(K(\sqrt{a},\sqrt{b})\big),\quad
		w+xi+yj+zij\mapsto
		\begin{pmatrix}
			w-x\sqrt{a} & y\sqrt{b}-z\sqrt{ab}\\
			y\sqrt{b}+z\sqrt{ab} & w+x\sqrt{a}
		\end{pmatrix}
\end{align}
	is an injective $F(\sqrt{-d})$-algebra
	homomorphism,
	and let us also write $\rho_\B^{-1}$
	to mean the inverse of the corestriction
	$\rho|^{\rho(\B)}:
		\B\overset{\cong}{\rightarrow}\rho_\B(\B):q\mapsto\rho_\B(q)$.
For future reference,
	this is
\begin{align}\label{rhoBinverse}
	\rho_\B^{-1}:\B\rightarrow
		\mathrm{M}_2\big(K(\sqrt a,\sqrt b)\big):\quad
		\begin{pmatrix}
			s & t\\
			u & v
		\end{pmatrix}\mapsto
		\frac{v+s}{2}+\frac{v-s}{2\sqrt a}i+\frac{u+t}{2\sqrt b}
			+\frac{u-t}{2\sqrt{ab}}.
\end{align}
Write the upper half-space model as the subspace
	$\HS^3=\RR\oplus\RR I\oplus\RR^+J$
	of Hamilton's quaternions,
	where $I^2=J^2=-1$
	and $IJ=-JI$.
Then 
\begin{align}\label{iota}
	\iota_\G:\HM_\G\rightarrow\HS^3,\quad
		w+xi+yj+\sqrt{-d} \, zij\mapsto
			\frac{y\sqrt{b} \, +z\sqrt{abd} \, I+J}{w+x\sqrt{a} \, }
\end{align}
	is an isometry such that the \Mob
	action $\G\times\HS^3\rightarrow\HS^3$
	is equal to $\iota
		\Big(\mu_\G\big(\rho_\A^{-1}(\cdot),\iota^{-1}(\cdot)\big)\Big)$.
That is,
	our quaternion representation $\mu_\G$
	of the group action from Corollary \ref{cor:macX}
	transfers to the usual \Mob
	action via $\iota$
	and $\rho_\B$.

\subsection{Quaternion Dirichlet domains}\label{sec:app sub:dir}

In this subsection we will see
	a shortcut to a method of computing Dirichlet domains
	introduced by Page \cite{Page2015},
	for Macfarlane manifolds.
We focus on Page's algorithm because earlier ones
	have been specific to the non-compact arithmetic case
	\cite{Hurwitz1881,Bianchi1892,Swan1971,Riley1983},
	or have required either arithmeticity 
	 \cite{CorralesJespersLealRio2004,Voight2009,Page2015}
	or compactness
	\cite{EpsteinPenner1988,Manning2002}
	whereas,
	as we showed in \S\ref{sec:man},
	Macfarlane manifolds include examples from every combination
	of arithmetic and non-arithmetic
	with compact and non-compact.
	
While Page's algorithm focuses on arithmetic
	Kleinian groups
	using their characterization within quaternion orders,
	it extends easily to arbitrary
	(complete orientable finite-covolume Kleinian)
	groups,
	as follows.
Given a finite set of matrix generators for the group,
	Page \cite[\S2.4.1]{Page2015}
	indicates an efficient way of listing the elements of the group
	in order of increasing Frobenius norm.
This has the effect of locating sides of a Dirichlet domain centered at
	$(0,0,0)$
	in the Poincar\'e ball model,
	in order of increasing distance from the center.

\subsubsection{\bf Page's algorithm translated to \boldmath{$\HM_\G$}}

Via the usual identification of the Poincar\'e ball model with $\HS^3$,
	the point $(0,0,0)$
	maps to $J\in\HS^3$,
	and $\iota_\G^{-1}$ from (\ref{iota})
	maps $J$
	to the point $1\in\HM_\G$,
	which is the canonical choice for our center.
So we will be computing the (canonical)
	Dirichlet domain $\Dir_\G(1)\subset\HM_\G$.

An isometry $\g\in\G$
	acts on $1$
	via $\mu_\G(\g,1)=\g\g\ct$,
	and we can interpret the Frobenius norm of $\g$
	as the square root of the trace of this image,
	as follows.
	
\begin{prop}\samepage\label{prop:Frob}
		If $\g=\mtxr\subset\kl$,
			then the trace of $\rho^{-1}_\B(\g)(1)$
			is
			\begin{equation*}
				|r|^2+|s|^2+|u|^2+|v|^2,
			\end{equation*}
			the square of the Frobenius norm of $\g$.
\end{prop}

\begin{proof}
	
	With $\g$
		as in the statement,
	\begin{align*}
		\rho_\B\big(\rho_\B^{-1}(\g)(1)\big)=\g\g\ct=\mtxr\begin{pmatrix}
			\overline{r} & \overline{u}\\
			\overline{s} & \overline{v}
		\end{pmatrix}
		=\begin{pmatrix}
			|r|^2+|s|^2 & r\overline{u}+s\overline{v}\\
			u\overline{r}+v\overline{s} & |u|^2+|v|^2
		\end{pmatrix}.
	\end{align*}
	Then $\tr\Big(\rho_\B^{-1}\big(\mu_\G(\g,1)\big)\Big)
		=\tr(\g\g\ct)
		=|r|^2+|s|^2+|u|^2+|v|^2$.
\end{proof}

The set of points in $\HM_\G$
	of some fixed trace $t$
	forms an ellipsoid,
	in particular
\begin{equation}\label{V_t}
	\bigg\{\frac{t}{2}+xi+yj+z\sqrt{-d}ij \;\; \bigg| \;\;
		x,y,z\in F,\ 4ax^2+4by^2+4abdz^2
			=t^2-4
		\bigg\}.
\end{equation}
Since $\G$
	is discrete,
	the set of points lying in each ellipsoid (which is compact)
	must be finite
	(Figure \ref{fig:Vt}
	illustrates a $2$-dimensional
	version of this).
\begin{figure}[h!]\centering
	\includegraphics[scale=0.8]{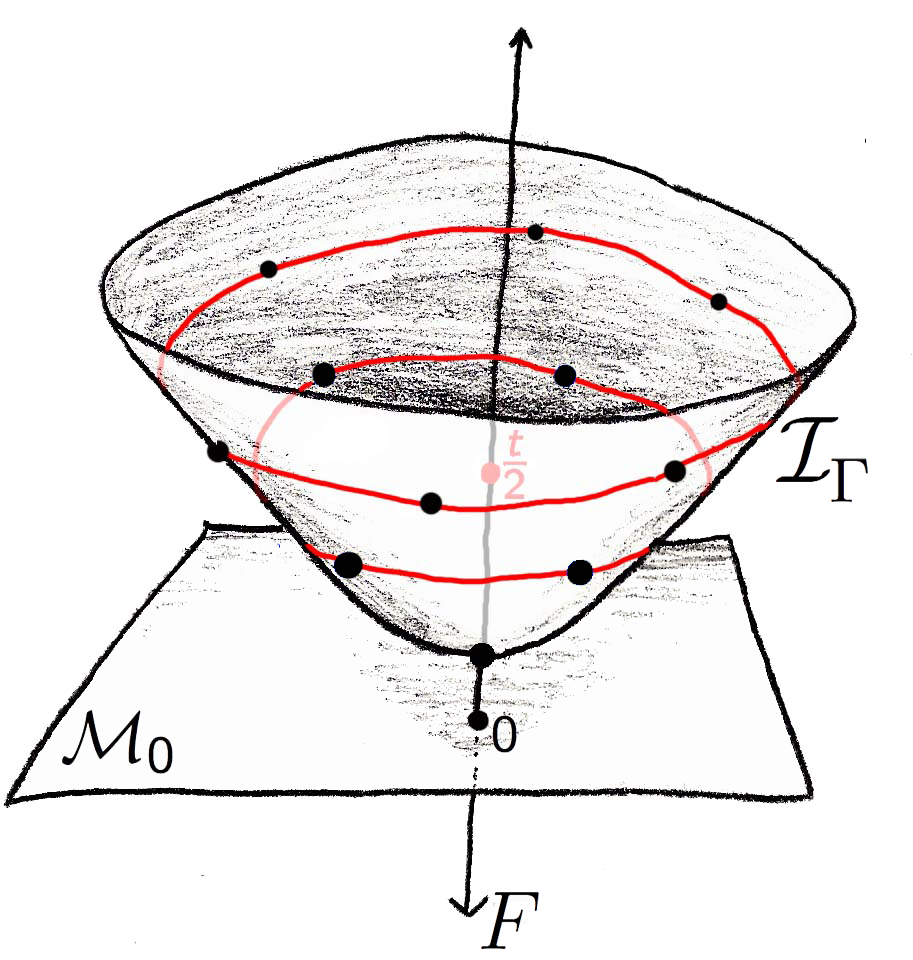}
	\caption{Orbit points by trace.}
	\label{fig:Vt}
\end{figure}
Thus,
	in our setup,
	Page's algorithm translates to searching for points in
	the orbit $\Orb_\G(1)=\{\g\g\ct\mid\g\in\G\}$
	in order of increasing trace,
	and computing the perpendicular bisector of each one
	until all sides of the Dirichlet domain $\Dir_\G(1)$
	have been found.
Details and efficiency analysis of this can be found in \cite{Page2015}.

\subsubsection{\bf Dual isometries and points}

We now explain our new contribution to the algorithm.
Due to the quaternionic structure,
	points on $\HM_\G$
	are also isometries of $\HM_\G$.
Likewise,
	elements of the group $\G$
	can occur among these points.
In particular,
	$\G\cap\HM_\G=\{\g\in\G\mid\g\ct=\g\}$.
These isometries take a special form that makes it especially easy
	to compute their perpendicular bisectors,
	as follows.
	
\begin{thm}\label{thm:GcapI}\samepage
	\begin{enumerate}
		\item[]
		\item
			The elements of $\G\cap\HM_\G\smallsetminus\{1\}$
				are precisely the purely hyperbolic isometries in $\G$
				that fix geodesics passing through $1$.
		\item
			For each $\g\in\G\cap\HM_\G$,
				the midpoint of between $1$
				and $\mu_\G(\g,1)$
				is $\g$.
		\item
			If $\G$
				is closed under complex conjugation,
				then $\G$
				contains every element of the orbit
				$\Orb_\G(1)\subset\HM_\G$.
	\end{enumerate}
\end{thm}

\begin{proof}
	Let $\g\in\G\cap\HM_\G$.
	Then $\n(\g)=1$
		and $\exists w,x,y,z\in F$
		such that $\g=w+xi+yj+z\sqrt{-d}ij$.
	Thus $w^2=1+ax^2+by^2+abdz^2$
		and since $a,b\in F^+$,
		we have that $w\in F^{\geq1}$.
	When $w=1$,
		it forces $\g=1$,
		and otherwise we have $\tr(\g)=2w\in F^{>2}$,
		making $\g$
		purely hyperbolic.
	
	Now let $\widetilde{g}$
		be the complete geodesic passing through $1$
		and $\g$,
		and we will prove that $\g$
		fixes this.
	Notice that (as illustrated in Figure \ref{fig:geo})
		the pure quaternion part of the points on $\widetilde{g}$
		are scalar multiples of the pure quaternion part $\g_0$
		of $\g$,
		so we can write
	\begin{equation}\label{geo}
		\widetilde{g}
			=\{q\in\M_+^1\mid \exists\ \lambda\in\RR:
				q_0=\lambda p_0\}.
	\end{equation}
	\begin{figure}[h]\centering
		\includegraphics[scale=0.8]{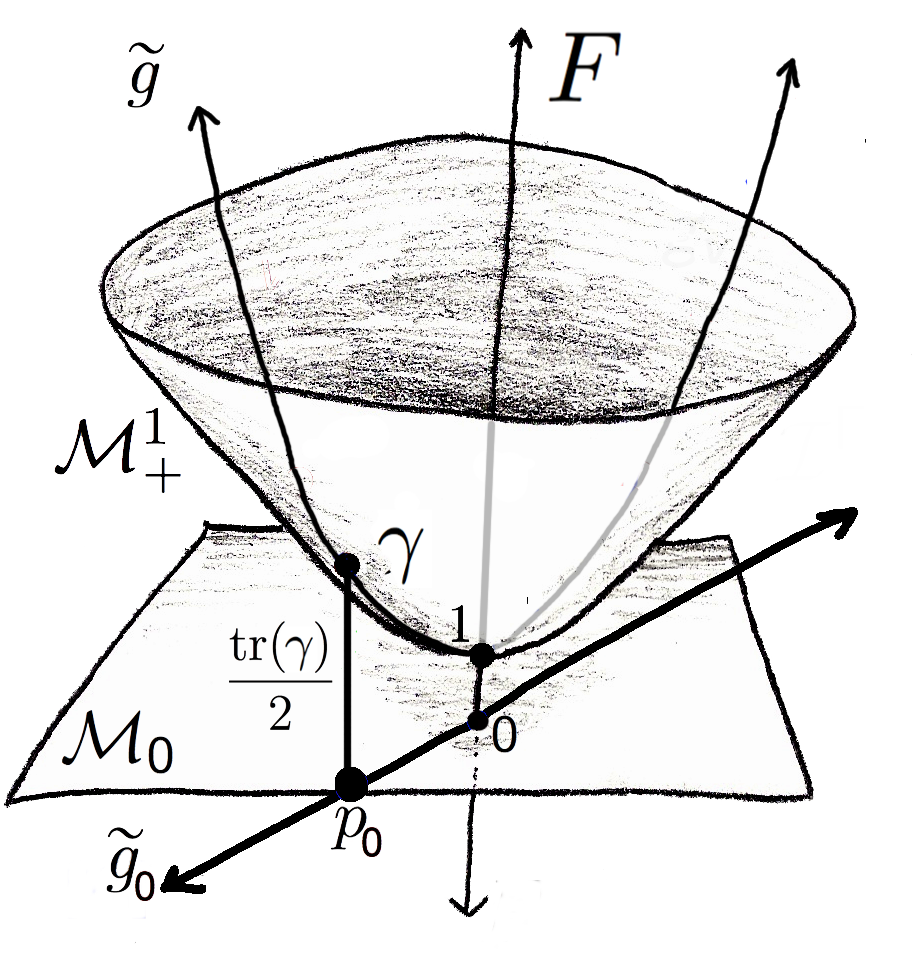}
		\caption{A $2$-dimensional analog of parametrizing
			$\widetilde g$
			using pure quaternions.}
		\label{fig:geo}
	\end{figure}
	Let $p=\frac{\tr(q)}{2}+\lambda p_0\in\widetilde{g}$
		in this notation,
		and then since $\g\in\M=\mathrm{Sym}(\B,\dagger)$,
		we have
	\begin{align*}
		\mu_\G(\g,q)&=\g q\g\ct=\g q\g=\bigg(\frac{\tr(\g)}{2}+\g_0\bigg)
			\bigg(\frac{\tr(q)}{2}+\lambda \g_0\bigg)
			\bigg(\frac{\tr(\g)}{2}+\g_0\bigg).
	\end{align*}
	If we multiply this out,
		there will be scalars $r,s,t,u\in F$
		so that the expression has the form
	\begin{align*}
		\mu_\G(\g,q)&=r+s\g_0+t\g_0^2+u\g_0^3\\
			&=(r+t\g_0^2)+(s+u\g_0^2)\g_0.
	\end{align*}
	But since $\g_0^2=\g_0(-\g_0^*)=-\n(\g_0)\in F$,
		this is indeed a point in $\widetilde{g}$
		as characterized by equation (\ref{geo}).
	
	Next we show that $\g$
		is the midpoint between $1$
		and $\mu_\G(\g,1)$.
	Let $\delta$
		be the hyperbolic translation along $\widetilde{g}$
		such that $\mu_\G(\delta,1)=\g$.
	Since $\g$
		is a hyperbolic translation along the same geodesic,
		it commutes with $\delta$,
		giving
		$\delta(\g)=\delta \g\delta\ct=\g\delta\delta\ct=\g\delta(1)=\g^2$.
	Therefore $\mathrm{d}(1,\g)$
		and $\mathrm{d}(\g,\g^2)$
		both equal the translation length of $\delta$.
	This proves (2),
		and also shows that every purely hyperbolic translation
		along a geodesic through $1$
		occurs as a point on $\HM_\G$,
		completing the proof of (1).
	
	For part (3),
		observe that if $\g,\g\ct\in\G$,
		then $\g(1)=\g\g\ct\in\G$.
\end{proof}

\begin{cor}
	If $\g\in\G\cap\HM_\G$,
		then $\widetilde{s}(\g)$
		 perpendicularly bisects $g(\g)$
		at $\g$,
		so that $\g$
		is the closest point to $1$
		on $\widetilde{s}(\g)$.
\end{cor}

We can use this to shorten the computation of $\Dir_\G(1)$
	by first approximating it by the bisectors
	of the elements of $\G\cap\HM_\G$,
	and we are guaranteed that any region outside of those
	does not lie in $\Dir_\G(1)$.
As we do this,
	we can skip the computation of their Frobenius norm,
	and just look through the matrices in $\G\cap\HM_\G$
	in order of increasing trace.
We can also skip the usually longer
	computation of perpendicular bisectors for these points,
	instead simply taking the hyperplane whose closest point to $1$
	is $\g\in\G\cap\HM_\G$.
As we will see in the examples at the end of this section,
	this often gives a lot of information so that completing the Dirichlet
	domain computation is easy from there.

This method is especially effective in the event that $\G$
	is closed under complex conjugation.
However it is important to note that even in that case,
	an element $\g$
	one finds by searching through $\G\cap\HM_\G$
	in order of increasing trace does not necessarily contribute
	a side to $\Dir_\G(1)$,
	even if $\widetilde s(\g)$
	(recalling the notation from \S\ref{sec:bg sub:DD})
	truncates the region computed so far.
In particular,
	if $\g\in\Orb_\G(1)$,
	then $\g=\delta\delta\ct$
	for some $\delta\in\G\smallsetminus\HM_\G$,
	in which case $\widetilde s(\delta)$
	passes halfway between $1$
	and $\g$.
(Notice that if $\g\in\Orb_\G(1)$
	and $\g=\delta\delta\ct$
	for some $\delta\in\G\cap\HM_\G$,
	then $\g$
	would not contribute a side since $\widetilde s(\g)$
	would be excluded from the region by $\widetilde s(\delta)$.)
On the other hand,
	if $\g\notin\Orb_\G(1)$
	and does contribute a side to the region computed thus far,
	then it also contributes a side to $\Dir_\G(1)$
	(at the very least,
	the point on that side which is closest to $1$
	will not be truncated by any other side).
	
So evoking Theorem \ref{thm:GcapI}
	when $\G$
	is closed under complex conjugation gives that
	each element of $\G\cap\HM_\G$
	either contributes a side to $\Dir_\G(1)$
	or lies on a side of $\Dir_\G(1)$
	at the point where that side is closest to $1$.
If one cannot see which case a given $\g$
	falls into by some easier means,
	it is straightforward to check whether $\g$
	factors into $\g=\delta\delta\ct$
	as a word in the generators.
This is will be illustrated in the examples at the end of this section.

\subsubsection{\bf Projective pure quaternions as slopes}

We gain an additional tool by expanding on the idea from the proof
	of Theorem \ref{thm:GcapI}
	where we used pure quaternion parts to characterize
	geodesics through $1$.
For each point $\mu_\G(\g,1)\in O_\G$,
	the geodesic ray that starts at $1$
	and passes through $\mu_\G(\g,1)$
	has a pure quaternion part which is a Euclidean ray
	(this can be seen in Figure \ref{fig:geo}).
So we can identify these geodesic rays as follows.

\begin{defn}\label{defn:slope}
	The \emph{slope} of $w+xi+yj+\sqrt{-d} \, zij\in O_\G$,
		is $[x,y,z]\in F^3/F^+$.
\end{defn}

As we search through $\G\cap\HM_\G$
	(or through $\Orb_\G(1)$)
	in order of trace,
	once we find an element having slope $[x,y,z]$,
	we know that any other element with that slope but higher trace
	cannot contribute a side to $\Dir_\G(1)$.
Afterall,
	for two bisectors of the same geodesic ray,
	one will be contained in the half-space of the other.
So keeping track of these slopes shortens the algorithm further
	because every time the slope of some element in $\G\cap\HM_\G$
	is the same as one already found,
	we know it will not contribute a side and can skip any
	more complicated means of determining this.

\subsection{Examples}\label{sec:app sub:ex}

We conclude by clarifying these ideas
	with some simple worked examples.
First we look at an implementation of the $2$-dimensional
	analogy of these ideas,
	then at an application to non-compact arithmetic
	manifolds.
Notably,
	the fundamental domains shown here 
	were drawn by hand in a basic illustration program,
	as the computational method was simple enough that it did
	not to require advanced software.

\subsubsection{\bf A Hyperbolic Punctured Torus}
	\label{sec:app sub:ex sub:PT}

Let $S$
	be the hyperbolic punctured torus.
Then $\pi_1(S)$
	can be represented by $\D<\mathrm{PSL}_2(\mathbb{Z})$
	where $\D$
	is the torsion-free subgroup of the modular group.
Then $\D=\langle\g,\delta\rangle$
	where
\begin{gather*}
	\g=\begin{pmatrix}1&1\\1&2\end{pmatrix}\quad\text{and}\quad
		\delta=\begin{pmatrix}1&-1\\-1&2\end{pmatrix}
		\text{\cite[p.116]{Bonahon2009}}.
\end{gather*}

The quaternion algebra of $\D$
	is $B\D
		=\Big(\frac{1,1}{\QQ}\Big)
		=\mathbb{Q}\oplus\QQ i\oplus\QQ j\oplus\QQ ij$
	where $i^2=j^2=1$ and $ij=-ji$,
	and let $\A=B\D$.
Then $\A$
	contains the restricted Macfarlane space $\caL
		=\mathbb{Q}\oplus\QQ i\oplus\QQ j$,
	and $\HM_\D:=\caL_+^1$
	is a quaternion hyperboloid model for $S$.
	
Recall that $\A\cong\mat_2(\QQ)$.
So the map (\ref{rhoB}) gives the $\QQ$-algebra
	isomorphisms
\begin{align*}
	\rho_\A&:
	\A\rightarrow\mathrm{M}_2(\QQ),\quad
		w+xi+yj+zij\mapsto\begin{pmatrix}
				w-x&y-z\\
				w+z&y+z
			\end{pmatrix},\\
	\rho_\A^{-1}&:
		\mathrm{M}_2(\QQ)\rightarrow\A,\quad
		\begin{pmatrix}
			s&t\\u&v
		\end{pmatrix}\mapsto
			\frac{v+s}{2}+\frac{v-s}{2}i+\frac{u+t}{2}j+\frac{u-t}{2}ij,
\end{align*}
	the map (\ref{iota})
	gives the isometry
\begin{align*}
	\iota_\D:\HM_\D\rightarrow\HS^2,\quad
		w+xi+yj\mapsto\frac{y+J}{w+x},
\end{align*}
	and these transfer the Macfarlane model to the \Mob
	action of $\D$
	on $\HS^2$.
When the context is clear,
	we use $\D=\langle\g,\delta\rangle$
	to mean both the matrix group
	and the corresponding quaternion group under $\rho_A^{-1}$,
	where
\begin{gather*}
	\g=\frac{3}{2}+\frac{1}{2}i+j
		\quad\text{and}\quad\delta=\frac{3}{2}+\frac{1}{2}i-j.
\end{gather*}

To implement the algorithm,
	we want to find the points in $\Orb_\D(1)$
	in order of increasing trace.
Since $\D$
	consists of all the non-elliptic elements of $\PSL_2(\ZZ)$,
	$\D$
	is closed under transposition,
	therefore
	$\Orb_\D(1)\subset\D$.
Then $\forall t\in\NN$
	$$\big\{q\in\Orb_\G(1)\;\big|\;\tr(qq\ct)=t\big\}
		\subset\big\{q\in\D\cap\HM_\D\mid\tr(q)=t\big\}.$$
But using $n(\D\cap\HM_\D)=1$,
	and the fact that
	hyperbolic elements have traces in $\RR\sm[-2,2]$,
	we can characterize these elements using a Diophantine equation
	where the only solutions for $t$
	lie in $\ZZ\sm\{0,\pm1\}$.
In particular,
\begin{align*}
	\big\{q\in\D\cap\HM_\D\;\big|\;\tr(q)=t\big\}
		=\bigg\{\frac{t}{2}+\frac{t-2x}{2}i+y \;\Big|\;  x^2+y^2=tx-1,
			\; x,y\in\ZZ\bigg\}.
\end{align*}

Once we know what this (finite) set is,
	we can find all matrices of trace $t$
	by checking wether each element can be written in the form
	$q=ww\ct$,
	for some word $w$
	in the generators of $\D$.
If it can,
	we get that $\widetilde{s}(q)$
	(in the notation of Definition \ref{defn:sides})
	passes halfway between $1$
	and $q$.
If it cannot,
	we get that $q^2\in\Orb_\D(1)$
	and $\widetilde{s}(q^2)$
	passes through $q$,
	by Theorem \ref{thm:GcapI}.


Table \ref{tab:puncT}
	lists some data from implementing this process,
	giving the points in $\D\cap\HM_\D$
	up to trace $18$.
For each $q\in\D\cap\HM_\D$,
	the direction of the corresponding geodesic ray
	is given by a normalized representative of the slope of $q$
	(in the sense of Definition \ref{defn:slope}). 
In the rightmost column,
	the corresponding points in $\HS^2$
	under $\iota_\D$
	are given.

Notice that (as predicted by Theorem \ref{thm:GcapI}),
	if a point $q$
	in the chart does not lie in $\Orb_\D(1)$,
	then later the point $q^2$
	does,
	and has the same slope.
For example,
	the elements of $\D\cap\HM_\D$
	at trace $3$
	lead to sides contributed by isometries as points at trace $6$.
The points in the table which lie in $\Orb_\G(1)$
	are in bold,
	and one can see that the sides they contributed had already
	been found in $\D\cap\HM_\D$
	before they were reached in the orbit.
		
Figure \ref{fig:puncT}
	shows in $\HS^2$
	which sides are contributed at each trace 
	(the table gives the trace of the isometry contributing the side,
	even though it was found earlier)
	until $\Dir_\D(1)$
	is complete, 
	and illustrates how the induced
	side-pairings create a punctured torus.

\subsubsection{\bf Non-compact Arithmetic Hyperbolic $3$-Manifolds}	
\label{sec:dir sub:noncomp arith}
As mentioned in \S\ref{sec:macs sub:non-compact sub:arith},
	a manifold in this class has a fundamental group which
	can be represented by a torsion-free finite index subgroup
	of a Bianchi group $\Bi$,
	$d\in\NN$ (square-free).
Let $\G$
	be such a group.
Then $B\G=\Big(\frac{1,1}{\QQ(\sqrt{-d})}\Big)$
	is Macfarlane and we can find a Dirichlet domain for $\G$
	by a similar method to that used in the previous subsection.
	
Since the only real traces occurring in $\Bi$
	lie in $\ZZ$,
	the ellipse at trace $t$
	can only be non-empty when $t\in\{2,3,4,\dots\}$.
Since $\Bi$
	is closed under complex conjugation,
	$\Orb_\G(1)\subset\Bi$,
	so like in the previous example,
	the ellipse at trace $t$
	is a subset of 
	$\big\{\g\in\Bi\mid\tr(\g)=t\big\}$.
The points in $\Bi$
	of trace $t$
	correspond to solutions to the Diophantine equation arising from
	$\det\begin{pmatrix}
		r & s\\
		\overline{s} & t-r
	\end{pmatrix}=1$
	where $r\in\ZZ$
	and $s\in\mathcal{O}_d$.
Once we find those,
	we can use the generators for $\G$
	to determine which lie in the orbit and which are intersected
	by sides.
	
\begin{ex}
	The fundamental group of the Whitehead Link Complement
		can be represented by the finite-index subgroup
		of $\PSL_2(\ZZ[\sqrt{-1}])$
		generated by $\begin{pmatrix}
				1 & 2\\
				0 & 1
			\end{pmatrix},
			\begin{pmatrix}
				1 & \sqrt{-1}\\
				0 & 1
			\end{pmatrix}$
			and $\begin{pmatrix}
				1 & 0\\
				-1-\sqrt{-1} & 1
			\end{pmatrix}$
		\cite[p.62]{MaclachlanReid2003}.
	Suppressing some details,
		an implementation of the process described above
		yields the Dirichlet domain for this group
		illustrated in Figure \ref{fig:WhiteheadDir},
		where we view the faces from above in $\HS^3$
		after applying $\rho_{B\G}$,
		and indicate the traces of the isometries contributing each side.
\end{ex}
\section{Acknowledgements}

This work has been made possible thanks to support from
	the Graduate Center of the City University of New York,
	and from the Instituto de Matem\'{a}ticas at
	Universidad Nacional Aut\'{o}noma de M\'{e}xico,
	Unidad de Cuernevaca.
I would like to thank
	Abhijit Champanerkar for helpful suggestions throughout,
	John Voight 
	for copious advice on properties of quaternion algebras
	and for making preprints of his book available,
	and the referees for their very constructive input.
I also thank
	Ian Agol,
	Marcel Alexander,
	Ben Blum-Smith,
	Neil Hoffman,
	Misha Kapovich,
	Stefan Kohl,
	Greg Laun,
	Ben Linowitz,
	Aurel Page,
	and Igor Rivine
	for additional helpful comments and correspondence.

\bibliographystyle{plain}
\bibliography{references}

\begin{table}[!htbp]
\centering
\caption{A punctured torus group intersected with its own
	quaternion hyperboloid model.}
\label{tab:puncT}
\begin{tabular}{| c || c | c | c | c |}
	\hline
		{trace} &
			{$q\in\HM_\D\cap\D$} &
			{slope of $q$} &
			{$\rho_\D(q)\in\mathrm{PSL}_2(\mathbb{\ZZ})$} &
			{$\iota_\D(q)\in\HS^2$}  \\
	\hline
		2 &
			1 &
			-- &
			$\begin{pmatrix}1&0\\0&1\end{pmatrix}$ &
			$J$ \\
	\hline
		3 &
			$\frac{3}{2}+\frac{1}{2}i\pm j$ &
			$[1,\pm2]$ &
			$\begin{pmatrix}1&\pm1\\\pm1&2\end{pmatrix}$ &
			$\pm\frac{1}{2}+\frac{1}{2}J$ \\ &
		$\frac{3}{2}-\frac{1}{2}i\pm j$ &
			$[-1,\pm2]$ &
			$\begin{pmatrix}2&\pm1\\\pm1&1\end{pmatrix}$ &
			$\pm1+J$ \\
	\hline
		6 &
			$3+2i\pm2j$ &
			$[1,\pm1]$ &
			$\begin{pmatrix}1&\pm2\\\pm2&5\end{pmatrix}$&
			$\pm\frac{2}{5}+\frac{1}{5}J$\\ &
			\boldmath$3-2i\pm2j$ & 
			$[-1,\pm1]$ & 
			$\begin{pmatrix}5&\pm2\\\pm2&1\end{pmatrix}$ &
			\boldmath$\pm2+J$ \\
	\hline
		7 &
			\boldmath$\frac{7}{2}+\frac{3}{2}i\pm3j$ & 
			$[1,\pm2]$ & 
			$\begin{pmatrix}2&\pm3\\\pm3&5\end{pmatrix}$ &
			\boldmath$\pm\frac{3}{5}+\frac{1}{5}J$ \\&
		$\frac{7}{2}-\frac{3}{2}i\pm3j$ &
			$[-1,\pm2]$ &
			$\begin{pmatrix}5&\pm3\\\pm3&2\end{pmatrix}$&
			$\pm\frac{3}{2}+\frac{1}{2}J$ \\
	\hline
		11 &
			\boldmath$\frac{11}{2}+\frac{9}{2}i\pm3j$ &
			$[3,\pm2]$ &
			$\begin{pmatrix}1&\pm3\\\pm3&10\end{pmatrix}$&
			\boldmath$\pm\frac{3}{10}+\frac{1}{10}J$ \\&
		$\frac{11}{2}-\frac{9}{2}i\pm3j$ &
			$[-3,\pm2]$ &
			$\begin{pmatrix}10&\pm3\\\pm3&1\end{pmatrix}$&
			$\pm3+J$ \\
	\hline
		15 &
			$\frac{15}{2}+\frac{11}{2}i\pm5j$ &
			$[11,\pm10]$ &
			$\begin{pmatrix}2&\pm5\\\pm5&13\end{pmatrix}$ &
			$\pm\frac{5}{13}+\frac{1}{13}J$ \\ &
		\boldmath$\frac{15}{2}-\frac{11}{2}i\pm5j$ &
			$[-11,\pm10]$ &
			$\begin{pmatrix}13&\pm5\\\pm5&2\end{pmatrix}$ &
			\boldmath$\pm\frac{5}{2}+\frac{1}{2}J$ \\ &
		$\frac{15}{2}+\frac{5}{2}i\pm7j$ &
			$[5,\pm14]$ &
			$\begin{pmatrix}5&\pm7\\\pm7&10\end{pmatrix}$ &
			$\pm\frac{7}{10}+\frac{1}{10}J$ \\ &
		$\frac{15}{2}-\frac{5}{2}i\pm7j$ &
			$[-5,\pm14]$ &
			$\begin{pmatrix}10&\pm7\\\pm7&5\end{pmatrix}$ &
			$\pm\frac{7}{5}+\frac{1}{5}J$ \\
	\hline
		18 &
			$9+8i\pm4j$ &
			$[2,\pm1]$ &
			$\begin{pmatrix}1&\pm4\\\pm4&17\end{pmatrix}$ &
			$\pm\frac{4}{17}+\frac{1}{17}J$ \\
		& $9-8i\pm4j$ & 
			$[-2,\pm1]$ &
			$\begin{pmatrix}17&\pm4\\\pm4&1\end{pmatrix}$ & 
			$\pm4+J$ \\ 
		& $9+4i\pm8j$ & 
			$[1,\pm2]$ &
			$\begin{pmatrix}5&\pm8\\\pm8&13\end{pmatrix}$ &
			$\pm\frac{8}{13}+\frac{1}{13}J$ \\
	  	& $9-4i\pm8j$ &
			$[-1,\pm2]$
			& $\begin{pmatrix}13&\pm8\\\pm8&5\end{pmatrix}$
			& $\pm\frac{8}{5}+\frac{1}{1}J$ \\
	\hline
\end{tabular}
\end{table}

\begin{figure}
\centering
	\caption{Dirichlet domain for a hyperbolic punctured torus.}
    \includegraphics[width=.95\textwidth]{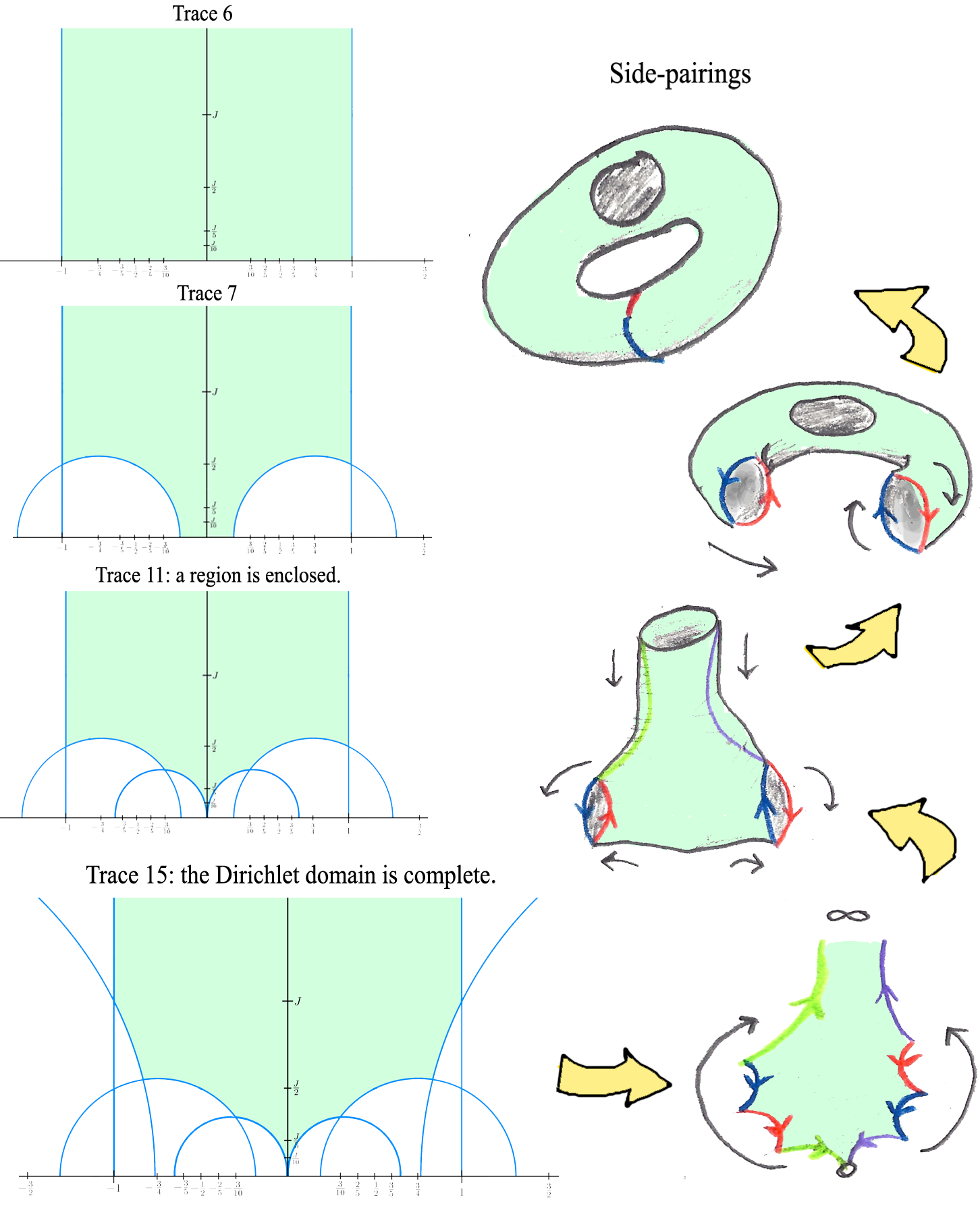}
	\label{fig:puncT}
\end{figure}

\begin{figure}
\centering
\caption{Dirichlet domain for the Whitehead link complement.}
	\includegraphics[width=1\textwidth]{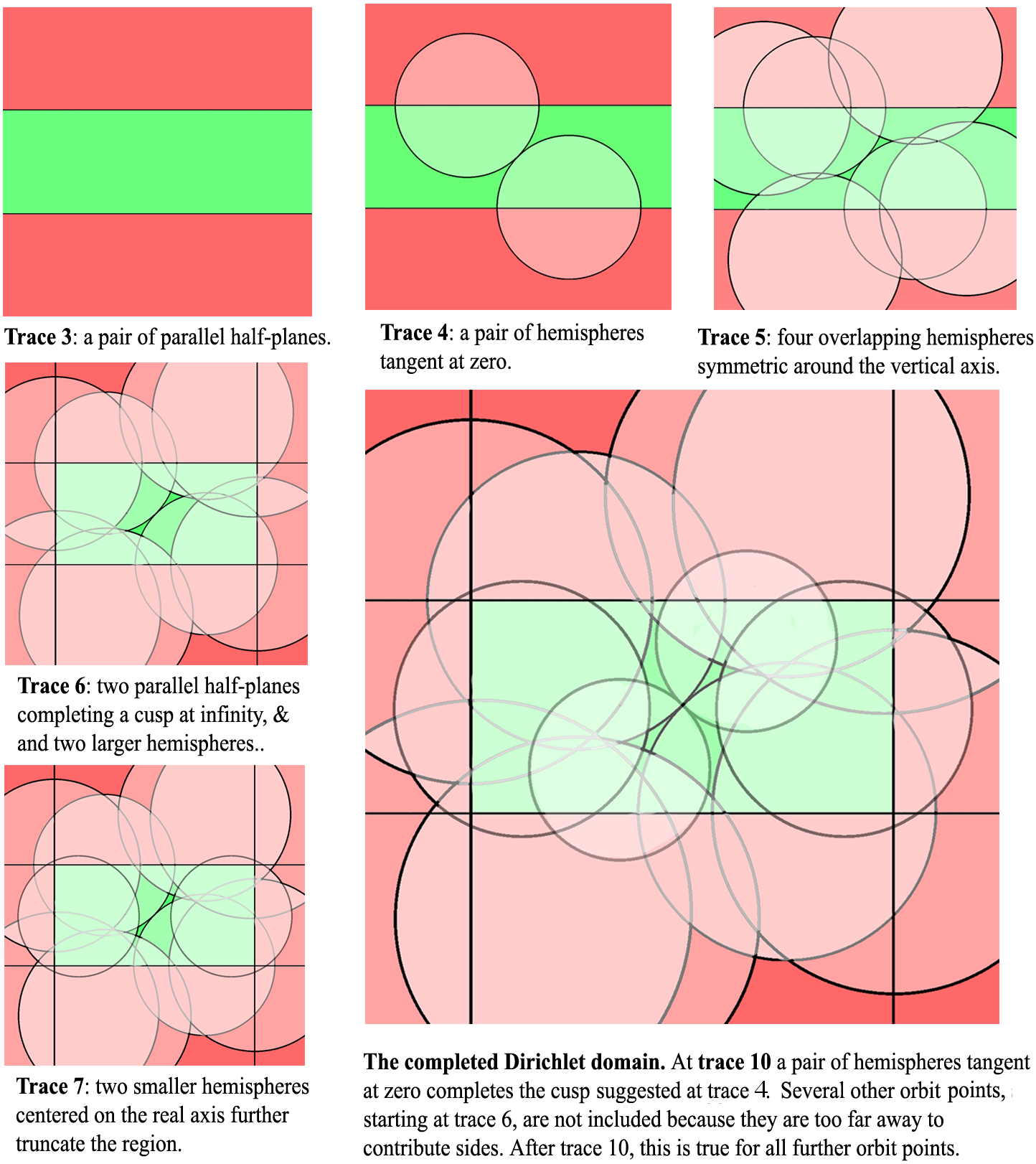}
	\label{fig:WhiteheadDir}
\end{figure}

\end{document}